\documentclass[10pt]{amsart}
\usepackage[T1]{fontenc}
\usepackage{amssymb}
\usepackage{amssymb,amsthm}
\usepackage{amscd,amssymb,graphics}

\usepackage[usenames]{color}
\usepackage{mathrsfs}

\usepackage{color}
\usepackage{amsfonts}
\usepackage{amsmath}
\usepackage{amsxtra}
\usepackage{latexsym}
\usepackage[mathcal]{eucal}
\usepackage{enumerate}
\usepackage{ifsym}
\usepackage{yfonts}
\usepackage{calc}

\usepackage{MnSymbol}

\usepackage[numbers]{natbib}

\usepackage[left=0.85in,right=0.85in,top=0.89in,bottom=1.15in]{geometry}
\input xy
\xyoption{all}

\numberwithin{equation}{section}

\newtheorem{thm}{Theorem}[section]
\newtheorem*{theoremA}{Theorem A}
\newtheorem*{theoremB}{Theorem B}
\newtheorem*{theoremC}{Theorem C}
\newtheorem*{theoremD}{Theorem D}

\newtheorem{fact}[thm]{Fact}
\newtheorem{corol}[thm]{Corollary}
\newtheorem{lemma}[thm]{Lemma}
\newtheorem{prop}[thm]{Proposition}

\newtheorem{quest}[thm]{Question}
\newtheorem{question}[thm]{Question}

\newtheorem{defi}[thm]{Definition}
\newtheorem{notation}[thm]{Notation}
\newtheorem{example}[thm]{Example}
\theoremstyle{definition}

\theoremstyle{remark}
\newtheorem{remark}[thm]{Remark}

\newcommand{\ben}{\begin{enumerate}}
	\newcommand{\een}{\end{enumerate}}
\newcommand{\bit}{\begin{itemize}}
	\newcommand{\eit}{\end{itemize}}

\def\R {{\mathbb R}}
\def\Q {{\mathbb Q}}

\def\N{{\mathbb N}}
\def\T{{\mathbb T}}

\def\Z {{\mathbb Z}}

\def\Aut{{\mathrm Aut}\,}

\def\H{{\mathcal H}}
\def\eps{{\varepsilon}}

\def\QED{\nobreak\quad\ifmmode\roman{Q.E.D.}\else{\rm Q.E.D.}\fi}

\def\HLM{hereditarily  locally minimal}
\def\H(L)M{hereditarily  (locally) minimal}
\def \HM {hereditarily   minimal}

\def \CHM {compactly hereditarily   minimal}
\def \CHLM {compactly hereditarily locally minimal}
\def \HM {hereditarily   minimal}

\def\hlm{\operatorname{HLM}}
\def\hm{\operatorname{HM}}
\def\chm{\operatorname{CHM}}
\def\chlm{\operatorname{CHLM}}

\begin{document}

\title{Hereditarily  minimal topological  groups}
\author[Xi]{W. Xi}
\address[W. Xi]{\hfill\break
	School of Mathematical Sciences
	\hfill\break
	Nanjing Normal University
	\hfill\break
	Wenyuan Road No. 1, 210046 Nanjing
	\hfill\break
	China}
\email{xiwenfei0418@outlook.com}

\author[Dikranjan]{D. Dikranjan}
\address[D. Dikranjan]{\hfill\break
Dipartimento di Matematica e Informatica
\hfill\break
Universit\`{a} di Udine
\hfill\break
Via delle Scienze  206, 33100 Udine
\hfill\break
Italy}
\email{dikranja@dimi.uniud.it}
\author[Shlossberg]{M. Shlossberg}
\address[M. Shlossberg]{\hfill\break
Dipartimento di Matematica e Informatica
\hfill\break
Universit\`{a} di Udine
\hfill\break
Via delle Scienze  206, 33100 Udine
\hfill\break
Italy}
\email{menachem.shlossberg@uniud.it}
\author[Toller]{D. Toller}
\address[D. Toller]{\hfill\break
Dipartimento di Matematica e Informatica
\hfill\break
Universit\`{a} di Udine
\hfill\break
Via delle Scienze  206, 33100 Udine
\hfill\break
Italy}
\email{daniele.toller@uniud.it}



\begin{abstract}
We study locally compact groups having all subgroups minimal. We call such groups {\it hereditarily} {\it minimal}. In 1972 Prodanov proved that the infinite hereditarily minimal compact abelian groups are precisely the groups $\Z_p$ of $p$-adic integers. We  extend  Prodanov's theorem to the non-abelian case at several levels. For infinite hypercentral (in particular, nilpotent) locally compact groups we show that the hereditarily minimal ones remain the same as in the abelian case. On the other hand, we classify completely the locally compact solvable \HM \ groups, showing that in particular they are always compact and metabelian.

The proofs involve the (hereditarily) locally minimal groups, introduced similarly. In particular, we prove a conjecture by He, Xiao and the first two authors, showing that the group $\Q_p\rtimes \Q_p^*$ is \HLM, where $\Q_p^*$ is the multiplicative group of non-zero $p$-adic numbers acting on the first component by multiplication. Furthermore, it turns out that the locally compact solvable \HM \ groups are closely related to this group.
\end{abstract}

\maketitle

\begin{center}
\today
\end{center}


\section{Introduction}

A Hausdorff topological group $(G, \tau)$ is called {\it minimal} if there exists no Hausdorff group topology on $G$ which is strictly coarser than $\tau$ (see \cite{Doitch,S71}).
This class of groups, containing all compact ones, was largely studied in the last five decades, (see the papers \cite{B,DS,DU,EDS,Meg,P}, the surveys \cite{DMe,S74} and the book \cite{DPS}). Since it is not stable under taking quotients, the following stronger notion was introduced in \cite{DP}:  a minimal group $G$ is {\it totally minimal}, if the quotient group $G/N$ is minimal for every closed normal subgroup $N$ of $G$.
This is precisely a group $G$ satisfying the open mapping theorem, i.e., every continuous surjective homomorphism with domain $G$ and codomain any Hausdorff topological group is open. Clearly, every compact group is totally  minimal.

Examples of locally compact  non-compact minimal groups can be found in \cite{RS}, they are all non-abelian. Indeed,  Stephenson \cite{S71}  noticed much earlier that local compactness and minimality jointly imply compactness, for abelian groups:

\begin{fact}\cite[Theorem 1]{S71} \label{Steph:Thm} A minimal locally compact abelian group is compact.
\end{fact}

In particular, $\R$ is not minimal. This failure of minimality to embrace also local compactness was repaired by Morris and Pestov \cite{MP}.  They called {\it locally minimal} a topological group $(G,\tau)$, having a neighborhood $V$ of the identity of $G$, such that for every coarser Hausdorff group topology $\sigma\subseteq \tau$ with $V\in \sigma$, one has $\sigma=\tau$.  Clearly, every minimal group is locally minimal. Moreover, every locally compact group is locally minimal.

Neither minimality nor local minimality are inherited by all subgroups (although they are inherited by closed central subgroups).  This justifies the following definition, crucial for this paper (these properties are abbreviated sometimes to  $\hm$ and $\hlm$ in the sequel):

\begin{defi}
A topological group $G$ is said to be {\em hereditarily} {\em (locally)}  {\em minimal}, if every subgroup of $G$ is  (locally) minimal.
\end{defi}

The following theorem of Prodanov provided an interesting characterization of the group of $p$-adic integers $\Z_p$ in terms of hereditary minimality.

\begin{fact}\label{TeoP}{\em \cite{P}}\label{fac:HMZ}
An infinite compact abelian group $K$ is isomorphic to $\Z_p$ for some prime $p$ if and only if  $K$ is \HM.
\end{fact}

Dikranjan and Stoyanov classified all \HM \ abelian groups.

\begin{fact}\label{DS-theorem}{\em \cite{DS}}
Let $G$ be a topological abelian group. Then the following conditions are equivalent:
\ben
\item each subgroup of $G$ is totally minimal;
\item  $G$ is \HM;
\item $G$ is topologically isomorphic to one of the following groups:\ben [(a)]
\item  a subgroup of $\Z_p$ for some prime $p$,
\item  a direct sum $\bigoplus F_p$, where for each prime $p$, the group $F_p$ is a finite abelian $p$-group,
\item $X\times F_p$,  where $X$ is a rank-one subgroup of $\Z_p$, and  $F_p$ is a finite abelian $p$-group.
\een \een
\end{fact}

Note that only the groups from item (a) can be  infinite locally compact. Indeed, using Fact \ref{Steph:Thm} one can extend Fact \ref{fac:HMZ} to locally compact abelian groups
as follows:

\begin{corol}\label{cor:spr}
An infinite \HM\ locally compact abelian group is isomorphic to $\Z_p$.
\end{corol}

The main aim of this paper is to extend this result to non-abelian groups  at various levels of non-commutativity. In particular, we obtain an extension to hypercentral (e.g., nilpotent) groups (Corollary \ref{cor:nprodanov}), yet some non-abelian groups may appear beyond the class of hypercentral groups (see Theorem C or Theorem D for a complete description in the case of solvable groups).
 Let us mention here that without any restraint on commutativity one can find examples of very exotic \HM \ groups even in the discrete case (see \S 3.1, entirely dedicated to discrete \HM \ groups and their connection to categorically compact groups).

 As far as \HLM \ groups are concerned, the following question  was raised in \cite[Problem 7.49]{DMe} in these terms:
{\em if  a connected locally compact group $G $ is \HLM, is $G$ necessarily a Lie group}?  Inspired by this question and the above results, hereditarily locally minimal groups were characterized in \cite{DHXX} among the locally compact groups which are either abelian or connected as follows:

\begin{fact}{\rm \cite[Corollary 1.11]{DHXX}}\label{fac:ext}
For a locally compact  group $K$ that is either abelian or connected, the following conditions are equivalent:
\ben[(a)]
\item $K$ is a \HLM \ group;
\item $K$ is either a Lie group or has an open subgroup isomorphic to $\Z_p$ for some prime $p.$\een
\end{fact}

This provides, among others, a characterization of the connected Lie groups as connected locally compact  \HLM \ groups.

\subsection{Main Results}
It was mentioned in \cite{DHXX}, that item (a) in Fact \ref{fac:ext} might not be equivalent to item  (b) in the non-abelian case, and it was conjectured (see \cite[Conjecture 5.1]{DHXX}) that a possible counter-example could be the group $(\Q_p,+)\rtimes \Q_p^*$, where $\Q_p^* = (\Q_p\setminus\{0\},\cdot)$ (here  the natural action of $\Q_p^*$ on $\Q_p$ by multiplication is intended). We prove that this conjecture holds true.

\begin{theoremA} Let $p$ be a prime. Then $(\Q_p,+)\rtimes \Q_p^*$  is \HLM. \end{theoremA}

We report in Theorem \ref{thm:retro} a result of Megrelishvili, ensuring that the group  in Theorem A is minimal.
Although it is not hereditarily minimal (e.g., its subgroup $(\Q_p,+)$ is not minimal by Fact \ref{Steph:Thm}), we show in  the classification Theorem D  that it contains (up to isomorphism) most locally compact solvable HM groups.

\medskip

Clearly, a \HM\ group is \HLM. In order to ensure the converse implication, we give the next definition.

\begin{defi}\label{def:CFN}
	For a topological group $G$ we consider the following properties:\\
	($\mathcal{N}_{fn}$) $G$ contains no finite normal non-trivial subgroups;\\
	($\mathcal{C}_{fn}$)
	Every infinite compact subgroup of $G$ satisfies ($\mathcal{N}_{fn}$).
\end{defi}

\begin{remark}\label{rem:pltor} Obviously, a  torsionfree group  $G$ satisfies both properties. Moreover, if $G$ is abelian, then $G$ is torsionfree if and only if it satisfies  ($\mathcal{N}_{fn}$).  It is also clear that  ($\mathcal{C}_{fn}$) implies ($\mathcal{N}_{fn}$) when $G$ is infinite compact.
\end{remark}

\begin{theoremB}\label{thm:herequiv}
	For a compact group $G$, the following conditions are equivalent:
	\ben [(a)]
	\item $G$ is a \HM \ group;
	\item $G$ is a \HLM \ group satisfying ($\mathcal{C}_{fn}$).
	\een
\end{theoremB}

\begin{remark}\label{new:rem} One cannot replace compact by locally compact as the groups $\Q_p$ and $\R\rtimes \Z(2)$
show (easy  counter-examples are provided also by arbitrary infinite discrete abelian groups).
\end{remark}

Our next result extends Corollary \ref{cor:spr}.

\begin{theoremC}\label{thm:tozp}
Let $G$ be an infinite \HM\  locally compact group that is either compact or locally solvable. Then $G$ is either center-free or isomorphic to $\Z_p$, for some prime $p$.
\end{theoremC}

In order to introduce our main result we need to first recall some folklore facts and fix the relevant notation.

\begin{notation} \label{new:notation}
For a prime $p$, let $\Z_p^*=(\Z_p\setminus p\Z_p,\cdot)$. Its torsion subgroup $F_p$ consists only of the $(p-1)$-th
 roots of unity if $p>2$, and $F_p = \{1,-1\}$ (the square roots of unity) if $p=2$. Moreover, it is cyclic and
\begin{equation*}
F_p \cong \begin{cases}
\Z(2) & \text{if } p = 2,\\
\Z(p-1) & \text{otherwise}.
\end{cases}
\end{equation*}
We denote by $C_{p}$ the subgroup of $\Z_p^*$ defined by
\begin{equation*}
C_{p}= \begin{cases}
1+4\Z_2 & \text{if } p = 2,\\
1+p\Z_p & \text{otherwise}.
\end{cases}
\end{equation*}

Finally, for $n\in \N$ we let $C_{p}^{p^n}=\{x^{p^n}: x\in C_p\} \leq C_p,$ so
\begin{equation}\label{Cppn:form}
C_{p}^{p^n}= \begin{cases}
(1+4\Z_2)^{2^n} = 1+2^{n+2}\Z_2 & \text{if } p = 2,\\
(1+p\Z_p)^{p^n} = 1+p^{n+1}\Z_p & \text{otherwise}.
\end{cases}
\end{equation}
\end{notation}

It is well known that $C_{p}\cong (\Z_p,+)$ and $\Z_p^* = C_p F_p \cong C_p \times F_p \cong (\Z_p,+) \times F_p$ as topological groups.

Now we provide two series of examples of hereditarily minimal metabelian locally compact groups that play a
prominent
role in our Theorem D:

\begin{example}\label{Exaaa}
Consider the natural action of $\Z_p^*$ on $(\Z_p,+)$ by multiplication, and the semidirect product
\[
K = (\Z_p,+)\rtimes \Z_p^*.
\]
Then $K$ is a subgroup of the group considered in Theorem A, so it is hereditarily locally minimal. Moreover, $K \cong (\Z_p,+)\rtimes \left( (\Z_p,+) \times F_p \right)$, so $K$ is compact and all of its non-abelian subgroups are minimal by Corollary \ref{cor:prec}. However, $K$ is not hereditarily minimal, as for example its compact abelian subgroup $\{0\} \rtimes \Z_p^* \cong \Z_p^* \cong (\Z_p,+) \times F_p$ is not hereditarily minimal by Fact \ref{TeoP}.
\begin{itemize}
\item[(i)] For a subgroup $F$ of $F_p$ and an integer $n\in \N$, consider the following subgroups of $K$:
\begin{gather*}
K_{p,F} = (\Z_p,+)\rtimes F \leq K_{p,F_p} \\
M_{p,n} = (\Z_p,+)\rtimes C_p^{p^n} \leq M_{p,0}.
\end{gather*}
For example, $K_{p,\{1\}} = (\Z_p,+)\rtimes \{1\}$ is isomorphic to $\Z_p$.
In Example \ref{ex:padic:rtimes:F}(a) we use a criterion for hereditary minimality of a  compact solvable group
(Theorem \ref{thm:charmeta}) in order to prove that $K_{p,F}$ is hereditarily minimal, while we use Theorem A and Theorem B to prove that $M_{p,n}$ is hereditarily minimal in Example \ref{padics:rtimes:padics}.

\item[(ii)] In case $p=2$ and $n\in \N$, we consider
also the group $T_n=(\Z_2,+)\rtimes_{\beta_n}C_2^{2^n}$ with the faithful action $\beta_n:C_2^{2^n}\times \Z_2 \to \Z_2 $ defined by
\begin{equation*}
\beta_n(y,x) =
\begin{cases}
yx & \text{ if } y\in C_2^{2^{n+1}},\\
-yx & \text{ if } y\in C_2^{2^{n}} \setminus C_2^{2^{n+1}}.
\end{cases}
\end{equation*}
(when no confusion is possible, we denote $\beta_n$ simply by $\beta$).
Note that the restriction $\beta' = \beta_n\restriction_{C_2^{2^{n+1}}\times \Z_2}:C_2^{2^{n+1}}\times \Z_2 \to \Z_2$ is the natural action by multiplication in the ring $\Z_2$, so $T_n \geq (\Z_2,+)\rtimes_{\beta'}C_2^{2^{n+1}} = M_{2,n+1}$. Obviously $[T_n:M_{2,n+1}]=2$, so $M_{2,n+1}$ is normal in $T_n$.
Another application of the criterion for hereditary minimality (Theorem \ref{thm:charmeta}) shows in Example \ref{ex:padic:rtimes:F}(b) that also the groups $T_n$ are \HM.
\end{itemize}
\end{example}

The following theorem classifies the locally compact solvable \HM\ groups by showing that these are precisely
the groups described in Example \ref{Exaaa}. Note that the only abelian ones among them are the groups $K_{p,\{1\}} = (\Z_p,+)\rtimes \{1\} \cong \Z_p$, for prime $p$.

\begin{theoremD}\label{thm:hmabc}
Let $G$ be an infinite locally compact solvable group, then the following conditions are equivalent:
	\ben
	\item  $G$ is \HM;
	\item $G$ is topologically isomorphic to one of the following groups:
	\ben [(a)]
		\item  $K_{p,F}  = \Z_p \rtimes F$, where $F\leq F_p$ for some prime $p$;
\item $M_{p,n}=\Z_p \rtimes C_p^{p^n}$, for some prime $p$ and  \ $n\in \N$;
\item $T_n=(\Z_2,+)\rtimes_{\beta}C_2^{2^n} $, for some $n\in \N$.
	\een \een
\end{theoremD}
A locally compact \HM\ group need not be compact in general, as witnessed by the large supply of infinite discrete \HM\ groups.
Nonetheless, as the groups in Theorem D are compact and metabelian, one has the following:

\begin{corol}
If $G$ is an infinite \HM\ locally compact solvable group, then $G$ is compact metabelian.
\end{corol}

Another nice consequence of Theorem D is the following: every closed non-abelian subgroup of $M_{p,n}$ is isomorphic to one of the groups in (b) or (c), while the closed abelian subgroups of $M_{p,n}$ are isomorphic to $\Z_p\cong K_{p,\{1\}}$.

The proof of  Theorem D is covered by Theorem \ref{thm:free} and Theorem \ref{prop:nonfree}, where we consider
 the torsionfree case and the non-torsionfree case, respectively.

Another application of  Theorem D is Theorem \ref{thm:htmkpf} in which  we classify the infinite locally compact, solvable, hereditarily totally minimal groups (see Definition \ref{def:htm}).

\begin{remark}\label{cor:compnofin}
Call a topological group $G$  {\em compactly hereditarily} {\em (locally)}  {\em minimal},  if every compact subgroup of $G$ is hereditarily (resp., locally) minimal.
We abbreviate it to $\chm$ and $\chlm$ in the sequel. For compact groups, being hereditarily minimal is equivalent to being compactly hereditarily minimal. Clearly, every discrete group is \CHM. Applying Theorem B to compact subgroups one can prove that
a topological group $G$ is  \CHM \  if and only if $G$ is a \CHLM \ group satisfying ($\mathcal{C}_{fn}$).
\end{remark}

\medskip
The next diagram summarizes some of the interrelations between the properties considered so far. The double arrows denote implications that always hold. The single arrows denote implications valid  under some additional assumptions.

\smallskip

\ \ \ \ \ \ \ \ \ \ \ \ \ \
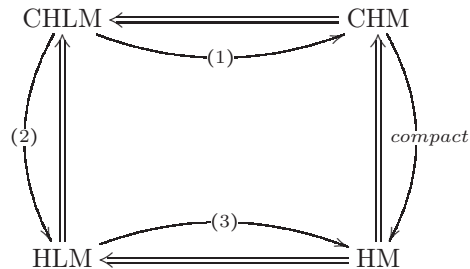
$${\xymatrix@!0@C4.2cm@R=3.2cm{
		\mbox{CHLM}
		\ar@/_1.2pc/|-{(2)}[d]
        \ar@/_1.2pc/|-{(1)
		}[r]
		&
		\mbox{CHM} \ar@{=>}[l]
		\ar@/^1.2pc/|-{\hspace{10pt} compact}[d]  &\\
		\mbox{HLM}
		\ar@{=>}[u]
		\ar@/^1.2pc/|-{(3)
		}[r]
		&
		\mbox{HM}
		\ar@{=>}[l]
		\ar@{=>}[u]
}}$$

\smallskip

\noindent (1): This implication holds true for
groups satisfying ($\mathcal{C}_{fn}$) (Remark \ref{cor:compnofin}). \\
(2): This implication holds true for totally disconnected locally compact groups (Proposition \ref{prop:ltdc}(1)).\\
(3): This implication holds true for compact
groups satisfying ($\mathcal{C}_{fn}$)  (Theorem B).

\smallskip

The group $\Q_{p}$ witnesses that the implication $\hm \Longrightarrow \chm$ \  cannot be inverted if compact is replaced by locally compact (and  totally disconnected). Indeed, $\Q_{p}$ is not minimal (so not \HM), yet every compact subgroup of $\Q_{p}$ is isomorphic to $\Z_p$, which means that $\Q_{p}$ is $\chm$. On the other hand, every non-trivial compact subgroup of the locally compact group $G=\R\times \Z_p$ is topologically isomorphic to $\Z_p$, so $G$ is $\chm$ by Prodanov's theorem. Yet $G$ is not \HLM \ by Fact \ref{fac:ext}.
This shows that none of the vertical arrows in the diagram can be inverted in general.

\bigskip

The paper is organized as follows. The proof Theorem A, given in \S\ref{Proof of Theorem A}, is articulated in several steps.
We first recall the crucial criteria for (local) minimality of dense subgroups (Fact \ref{Crit}). Using these criteria, we show in \S\ref{sub:nab} that every non-abelian subgroup $H$ of $L$ is locally minimal, while  \S \ref{sub:main} takes care of the abelian subgroups of $L$.

 Section \S\ref{lcHMSection} contains some general results on locally compact HM groups
and the proof of Theorem B.
In \S\ref{discreteHM} we provide a brief review on the relevant connection between discrete categorically compact groups
and the discrete HM groups. 
In \S \ref{subsection:proofB} we give some general results on non-discrete locally compact HM groups, proving in Proposition \ref{prop:ndhmlc} that they are totally disconnected, and contain a copy of $\Z_p$ (in particular, an infinite locally compact HM group is torsion if and only if it is discrete, and in this case it is not locally finite).
Furthermore, such a group $G$ satisfies ($\mathcal{N}_{fn}$),
so that either $Z(G) = \{e\}$, or $Z(G) \cong \Z_p$ for a prime $p$. In the
latter case, $G$ is also torsionfree by Corollary \ref{cor:pgr}. So a non-discrete locally compact HM group is either center-free or torsionfree.
Finally, we prove Theorem B (making use of Proposition \ref{pro:tor}) and apply Theorem B to see in Example \ref{padics:rtimes:padics} that the groups $M_{p,n}$ are HM.

 In \S\ref{lcHM} we explore  infinite non-discrete locally compact HM groups with non-trivial center,
proving in Corollary \ref{cor:nprodanov} that the hypercentral ones are isomorphic to $\Z_p$ and we give a proof of Theorem C.
In Theorem \ref{add:prop:thm} we collect some necessary conditions a non-discrete locally compact HM group with non-trivial center must satisfy.

 In \S\ref{Semidirect products of p-adic integers} we prepare the tools for the proof of Theorem D, by proving that the groups introduced in Example \ref{Exaaa} are pairwise non-isomorphic (Corollary \ref{thm:pair} and propositions \ref{prop:p=2} and \ref{Tn-Mpn:non-isom}) and by classifying the semidirect products of $\Z_p$ with some compact subgroups of $\Z_p^*$ (Proposition \ref{prop:alpha} and Lemma \ref{lem:uonique}).

 Theorem D is proved  in \S\ref{Proof of Theorem D}. To this end we first provide a criterion Theorem \ref{thm:charmeta}, used in Example \ref{ex:padic:rtimes:F} to check that the groups $K_{p,F}$ are HM. Another consequence of Theorem \ref{thm:charmeta} is Lemma \ref{laaaasts:lemma}, that we apply to
show that the groups $T_n$ are HM in Example \ref{ex:Tn}.
We start the proof of Theorem D in \S\ref{The general case}, by proving some reduction results
(Proposition \ref{prop:ms}), and some general results (Propositions \ref{prop:tri}, \ref{prop:abc} and
\ref{prop:dich}).
 In \S\ref{Torsionfree case} we consider the torsionfree case of Theorem D in Theorem \ref{thm:free}, while the non-torsionfree case Theorem \ref{prop:nonfree} is proved in \S\ref{Non-torsionfree case}, based on the technical result Proposition \ref{prop:sknmin} dealing with the case $p=2$.

We dedicate \S\ref{Hereditarily  totally minimal topological  groups} to  hereditarily totally minimal groups  (HTM for short, see Definition \ref{def:htm}). The only locally compact solvable ones to consider are the groups classified in Theorem D, and we first prove in Proposition \ref{prop:kpfhtm} that
 the groups $K_{p,F}$ are HTM.  Then we see in Proposition \ref{prop:mpntn} that the HM groups $M_{p,n}$ and $T_n$ are not HTM, leading us to Theorem \ref{thm:htmkpf}, which describes the groups $K_{p,F}$ as the only infinite locally compact solvable HTM groups.

The last \S\ref{Open questions and concluding remarks} collects some open questions, a partial converse to Theorem \ref{add:prop:thm}, and some final remarks.

\bigskip

\subsection{Notation and terminology}
We denote by $\Z$  the group of integers, by $\R$ the real numbers, and by  $\N$ and $\N_{+}$ the non-negative integers and positive natural numbers, respectively. For $n\in \N_+$, we denote by $\Z(n)$ the finite cyclic group with $n$ elements.
If $p$ is a prime number, $\Q_p$ stands for the field of $p$-adic numbers, and $\Z_p$ is its subring of $p$-adic integers.

Let $G$ be a group. We denote by $e$ the identity element.
If $A$ is a non-empty subset of $G$, we denote by $\langle A\rangle$ the subgroup of $G$ generated by $A$. In particular, if $x$ is an element of $G$, then $\langle x\rangle$ is a cyclic subgroup. If $F=\langle x\rangle$ is finite, then $x$ is called a \emph{torsion} element, and $o(x)= |F|$ is the \emph{order} of $x$. We denote by $t(G)$ the torsion part of the group $G$ and $G$ is called \emph{torsionfree} if $t(G)$ is  trivial. The {\it centralizer} of $x$  is $C_G(x)$.  If the {\it center} $Z(G)$ is trivial, then we say that $G$ is \emph{center-free}. A group $G$ is called  \emph{$n$-divisible} if $nG=G$ for $n\in \N_+$.

Let $\mathcal P$ be an algebraic (or set-theoretic) property.  A group is called \emph{locally $\mathcal P$} if every finitely generated subgroup has the property $\mathcal P$. For example, in a locally finite group every finitely generated subgroup is finite.

The \emph{$n$-th center} $Z_n(G)$ is defined as follows for $n\in \N$. Let $Z_0(G) = \{e\}$,  $Z_1(G) = Z(G)$, and assume that $n > 1$ and $Z_{n-1}(G)$ is already defined. Consider the canonical projection $\pi \colon G \to G/Z_{n-1}(G)$ and let $Z_n(G) = \pi^{-1} Z (G/Z_{n-1}(G) )$. Note that
$Z_n(G)=\{x\in G: [x,y]\in Z_{n-1}(G)  \text{ for every } y\in G\}$. This produces an ascending chain of subgroups $Z_n(G)$ called the upper central series of $G$, and a group is {\em nilpotent} if $Z_n(G) = G$ for some $n\in \N$. In this case, its nilpotency class is the minimum of such $n$. For example, the groups with nilpotency class at most $1$ are the abelian groups.
One can continue the upper central series to infinite ordinal numbers via transfinite recursion: for a limit ordinal $\lambda$, define $Z_{\lambda }(G)=\bigcup _{\alpha <\lambda } Z_{\alpha }(G)$.
A group is called {\it hypercentral} if it coincides with $Z_{\alpha }(G)$ for some ordinal $\alpha$.
Nilpotent groups are obviously hypercentral, while  hypercentral groups are locally nilpotent.

We denote by $G'=G^{(1)}$ the {\it derived subgroup} of $G$, namely the subgroup of $G$ generated by all commutators $[a,b]=aba^{-1}b^{-1},$ where $a,b\in G.$ For $n\geq 1$, define $G^{(n)}=(G^{(n-1)})'$ and also  $G^{(0)}=G$.  We say that $G$ is {\em solvable} of class $n$ for some $n\in \N$, if $G^{(n)}=\{e\}$ and $G^{(m)}\ne\{e\}$ for $0\leq m<n$. If $G$ is solvable of class $n$, then $G^{(n-1)}$ is abelian. In particular, $G$ is  {\em metabelian}, if $G$ is solvable of class at most $2$.

For an integral domain (in particular, a field) $A$ we denote by $A^*$ the multiplicative group of all invertible elements of $A$ (resp., the group $(A\setminus\{0\},\cdot)$).

All the topological groups in this paper are assumed to be Hausdorff. For a topological group $G$, the connected component of $G$ is denoted by $c(G)$.   For a subgroup $H\leq G$, the closure of $H$ is denoted by $\overline{H}.$
A topological group is {\emph{precompact} if it is isomorphic to a subgroup of a compact group. Let $S$ and $T$ be topological groups and $\alpha:S\times T\to T$ be a continuous action by automorphisms.  We say that the action $\alpha$ is \emph{faithful}
		if $\ker\alpha=\{s\in S: \forall t\in T \ \alpha(s,t)=t \}$ is trivial.

All unexplained terms related to general topology can be found in \cite{En}. For background on abelian groups, see \cite{Fuc}.

\section{Proof of Theorem A}\label{Proof of Theorem A}
There exist useful criteria for establishing the minimality (local minimality) of a dense subgroup of a minimal (respectively, locally minimal) group. These criteria are based on the following definitions.
\begin{defi}
Let $H$ be a subgroup of a topological group $G$.
\ben
	\item \cite{P,S71} $H$ is said to be {\it essential} in $G$ if $H\cap N\neq \{e\}$ for every non-trivial closed normal subgroup $N$ of $G.$
	\item \cite{ACDD} $H$ is  {\it locally essential} in $G$ if there exists a neighborhood $V$  of $e$ in $G$ such that $H\cap N\neq \{e\}$ for every non-trivial closed normal subgroup $N$ of $G$ which is contained in $V.$\een
\end{defi}

\begin{fact}\label{Crit} Let $H$ be a dense subgroup of a topological group $G.$
\ben \item  \cite[Minimality Criterion]{B}  $H$ is minimal if and only if $G$ is minimal and $H$ is essential in $G$ (for compact $G$ see also \cite{P,S71}).
\item  \cite [Local Minimality Criterion]{ACDD}  $H$ is locally  minimal if and only if $G$ is locally minimal and $H$ is locally essential in $G.$ \een
\end{fact}

In this section, $L$ denotes the group $(\Q_p,+)\rtimes \Q_p^*$.

 In \S\ref{sub:nab}, we show that every non-abelian subgroup $H$ of $L$ is essential (in particular, locally essential) in its closure $\overline{H}.$ Since the latter group is locally compact (and thus locally minimal), we conclude by the above criterion that $H$ is also locally minimal. At the same time, we deduce by the above  Minimality Criterion that $H$ is minimal if and only if $\overline{H}$ is minimal in Corollary \ref{thm:dense}. In particular, every precompact non-abelian subgroup of $L$ is minimal. Finally, using Fact \ref{fac:ext},  we prove in \S \ref{sub:main} that also the abelian subgroups of $L$ are locally minimal.

\subsection{Non-abelian subgroups of $L$}\label{sub:nab}

We begin this section with some easy lemmas of independent interest.
\begin{lemma}\label{lem:tre}

If $G$ is a group, then $G'$ non-trivially meets every normal subgroup $N$ of $G$ not contained in $Z(G)$.

In particular, if $G$ is a topological group with trivial center, then $G'$ is essential in $G$.

\end{lemma}
\begin{proof}

Let $a\in N \setminus Z(G)$,
and let $b\in G$ be such that $e \neq [a,b]=aba^{-1}b^{-1}$. Then $[a,b]\in N\cap G'$, as $N$ is normal.
\end{proof}

\begin{lemma} \label{lem:abel}
	If $H$ is a non-abelian subgroup of a group $G$, then $H\cap G'$ is non-trivial.
\end{lemma}
\begin{proof}
	If $a,b$ are non-commuting elements of $H$, then $\{e\} \neq [a,b]\in H\cap G'.$
\end{proof}

\begin{lemma}\label{lem:cyc}
Every non-trivial subgroup $H$ of $\Q_p$ is essential.
\end{lemma}
\begin{proof}
First note that an element $e \neq x \in \Q_p$ has the form $x = p^n a$ for $a \in \Z_p^*$, and $n \in \Z$,
so  its generated subgroup is $\langle x \rangle = x \Z = p^n a \Z$ (note that here $\Z$ carries the $p$-adic topology, induced by $\Q_p$).

Let $N$ be a non-trivial closed subgroup of $\Q_p$, and let $H \geq p^n a \Z$, and $N \geq p^m b \Z$ for some
$a, b \in \Z_p^*$ and $n, m \in \Z$.
Being closed, $N$  also contains the subgroup $p^m b \Z_p$, which coincides with $p^m \Z_p$.
Then $e \neq a p^{ \max \{n,m\} } \in N \cap H$.
\end{proof}

\begin{lemma}\label{lem:trabel}
If $H$ is a non-abelian subgroup of $L$, then $Z(H)$ is trivial.
\end{lemma}
\begin{proof}
Let $(m,n)\in Z(H)$, and we  show that $m=0$ and $n=1$.

Let $(a,b), (c,d)\in H$ be non-commuting elements.
Then $(m,n)$ commutes with both $(a,b)$ and $(c,d)$ while the latter two elements do not commute with each other. This implies the following:
\ben \item $a(1-n)=m(1-b)$,
	 \item $c(1-n)=m(1-d)$,
	 \item $a(1-d)\neq c(1-b)$.\een
Multiply $(1)$ by $c$ to obtain $ac(1-n)=mc(1-b).$ This together with $(2)$ imply that $am(1-d)=mc(1-b).$ In view of $(3),$ the latter equality is possible only if $m=0.$ Using $(1)-(2)$   we now obtain $a(1-n)=c(1-n)=0.$ By $(3)$, either $a \neq 0$ or $c \neq 0$, and thus $n=1.$
\end{proof}
	
\begin{prop}\label{nonab:subgr:essent}
If $H$ is a non-abelian subgroup of $L$, then $H$ is essential in $\overline{H}$.
\end{prop}
\begin{proof}
We first consider the subgroup $H_1 = H \cap L'$, which is non-trivial by Lemma \ref{lem:abel}. As $L'=\Q_p\rtimes \{1\}$ is isomorphic to $\Q_p$, Lemma \ref{lem:cyc} implies that $H_1$ is essential in $L'$.

Now let $N$ be a non-trivial closed normal subgroup of $\overline{H}$, and we have to prove that $N \cap H$ is non-trivial.

Since $\overline{H}$ is non-abelian, its center is trivial by Lemma \ref{lem:trabel}, so $\overline{H}'$ is essential in $\overline H$ by Lemma \ref{lem:tre}. Then $N\cap \overline{H}'$ is non-trivial, and, in particular, $N_1 = N \cap L'$ is non-trivial.  Obviously, $N$ is closed in $L$, so $N_1$ is a closed subgroup of $L'$. The essentiality of $H_1$ in $L'$ gives that $N_1 \cap H_1$ is non-trivial, so  also $N \cap H$ is non-trivial.
\end{proof}
	
\begin{corol}\label{thm:dense}
Let $H$ be a non-abelian subgroup of $L$. Then:
	\ben \item $H$ is locally minimal;
	\item $H$ is minimal if and only if $\overline{H}$ is minimal.\een
\end{corol}
\begin{proof}
$(1)$: Since $\overline{H}$ is locally compact, and thus locally minimal, we can apply Proposition \ref{nonab:subgr:essent} and the Local Minimality Criterion.

$(2)$: Apply the Minimality Criterion.
\end{proof}

Since $L$ is complete, and compact groups are minimal, we immediately obtain the following consequence of Corollary \ref{thm:dense}(2).
\begin{corol}\label{cor:prec}
If $H$ is a precompact non-abelian subgroup of $L$,
then $H$ is minimal.
\end{corol}

Recall that a group $G$ is \emph{perfectly minimal} if $G \times M$ is minimal for every minimal group $M$. The above result should be compared with the following Megrelishvili's theorem, ensuring that $L$ is perfectly minimal.

Let $K$ be a topological division ring. A subset $B$ of $K$ is called \emph{bounded} if for every neighborhood $X$ of $0$ there is a neighborhood $Y$ of $0$ such that $YB \subseteq X$ and $BY \subseteq X$. A subset $V$ of $K$ is \emph{retrobounded} if $0 \in V$ and $(K \setminus V)^{-1}$ is bounded.
Then $K$ is called \emph{locally retrobounded} if it has a local base at $0$ consisting  of \emph{retrobounded} neighborhoods.
For example, $K$ is locally retrobounded if it is: locally compact,
topologized by an absolute value,
or a linearly ordered field.

\begin{fact}\cite[Theorem 4.7(a)]{Meg}\label{thm:retro}
Let $K$ be a non-discrete locally retrobounded division ring.
Then the group $G = K \rtimes K^*$
is perfectly minimal, where the natural action of $K^* $ on $K = (K,+)$ by multiplication is considered.
\end{fact}

As $K = \Q_p$ is locally retrobounded, Fact \ref{thm:retro} entails that $L$ is perfectly minimal. By Lemma \ref{lem:trabel}, the non-abelian subgroups of $L$ are center-free, so we can apply the above results to obtain the following corollary.

\begin{corol}\label{prod:nonAB:subgr:ofL}
For every $i\in I$, let $H_i$ be a non-abelian subgroup of $L$ that is either dense or precompact. Then, the product $\displaystyle{\prod}_{i\in I} H_i$ is perfectly minimal.
\end{corol}
\begin{proof}
If $H_i$ is precompact,
then it is minimal by Corollary \ref{cor:prec}. If $H_i$ is dense in $L$, then it is minimal by Fact \ref{thm:retro} and Corollary \ref{thm:dense}(2). Thus, $H_i$ is minimal  for every $i\in I.$ These subgroups are center-free  according to Lemma \ref{lem:trabel}. As the arbitrary product of center-free minimal groups is perfectly minimal by \cite[Theorem 1.15]{Meg95}, we conclude that $\displaystyle{\prod}_{i\in I} H_i$ is perfectly minimal.
\end{proof}

\subsection{Abelian subgroups of $L$} \label{sub:main}

\begin{lemma}\label{last:lemma}
An infinite compact subgroup $C$ of $G = \Q_p^*$ contains an open subgroup of $G$ isomorphic to $\Z_p$.
\end{lemma}
\begin{proof}
Note that $G \cong \Z \times \Z_p^*$, where $\Z$ is equipped with the discrete topology, so  $G\cong \Z  \times F_p \times   \mathbb Z_p$, and we identify these two groups.

Being infinite compact, $C$ non-trivially meets the open subgroup $U = \{0\} \times \{0\} \times \mathbb Z_p$ of $G$, otherwise the projection of $C$ in the discrete quotient group $G/U$ would be compact infinite.
So consider the non-trivial subgroup $O = C\cap U$ of $C$.
As $O$ is a closed subgroup of $U \cong \Z_p$,
it is isomorphic to $\Z_p$ itself and it has finite index in $U$, so it is also open in $U$. As $U$ is open in $G$, we conclude that $O$ is open in $G$.
\end{proof}

Now we are in position to prove
Theorem A.

\medskip
\noindent{\bf Proof of Theorem A.}
If $H$ is a non-abelian subgroup of $L$, then Corollary \ref{thm:dense}(1) applies.
Let $H$ be an abelian subgroup of $L$, and consider the following two possibilities:

\smallskip

Case 1:  there is  $e \neq h\in H \cap L' $, then obviously $H \leq C_L(h)$.
Since $L' = (\Q_p,+) \rtimes \{1\}$,
 we have $C_L(h) = L'$. It follows that $H \leq L'\cong \Q_p$. Then $H$ is locally minimal by Fact \ref{fac:ext}.

\smallskip

Case 2:  the subgroup $H \cap L' $ is trivial, so the projection $L \to L/L'$ restricted to $H$ gives a continuous group isomorphism $q: H \to q(H)$.

It is not restrictive to assume $H$ to be non-discrete, and if we prove that the closure of $H$ in $L$ has an open subgroup isomorphic to $\Z_p$, it will follow by Fact \ref{fac:ext} that $H$ is locally minimal; so we can assume $H$ to be closed in $L$.

Then $H$ is a non-discrete locally compact and totally disconnected group, so $H$ contains an infinite compact open subgroup $K$ by van Dantzig's theorem \cite{vD}, and the restriction $q\restriction_{ K }: K \to q(K)$ is a closed map, hence a topological group isomorphism. The infinite compact subgroup $q(K)$ of $q(H) \leq L/L' \cong \Q_p^*$
contains an open subgroup $O$ isomorphic to $\Z_p$ by Lemma \ref{last:lemma}.
Obviously, also $q\restriction_{ K }^{\phantom{K}-1} (O)$ is isomorphic to $\Z_p$, and it is open in $K$, hence in $H$. \qed

\section{Locally compact $\hm$ groups}\label{lcHMSection}

In this section we consider hereditarily minimal locally compact groups $G$. We start with the following immediate consequence of Corollary \ref{cor:spr}.

\begin{lemma}\label{Prod+Steph:lemma}
Let $G$ be a hereditarily minimal locally compact group, and let $A$ be a closed abelian subgroup of $G$. Then either $A$ is finite, or $A \cong \Z_p$ for some prime $p$.

\end{lemma}
 In particular, if $G$ is not a torsion group, then it contains a copy of some $\Z_p$,
so $G$ is not discrete.

Now we give some results on the discrete \HM\ groups.
\subsection{Discrete $\hm$ groups}\label{discreteHM}
Clearly, every locally finite group is torsion. In the following result we recall that the converse holds true for locally solvable groups.

\begin{fact}\cite[Proposition 1.1.5]{DIX}\label{fac:lslf}
	Every torsion locally solvable group is locally finite.
\end{fact}

The next fact guarantees the existence of an infinite abelian group inside every infinite locally finite group.
\begin{fact}\cite{HK}\label{fac:lfia}
	Every infinite locally finite group contains an infinite abelian subgroup.
\end{fact}

From Fact \ref{fac:lslf} and Fact \ref{fac:lfia}   we obtain the following corollary.
\begin{corol}\label{cor:lsia}
	An infinite locally solvable group contains an infinite abelian subgroup.
\end{corol}

The next result will be used in the sequel to conclude that an infinite locally solvable \HM\ group is not discrete.
\begin{lemma}\label{lem:tnlf}
If $G$ is an infinite discrete \HM\ group, then the abelian subgroups of $G$ are finite.

In particular, the center of $G$ is finite, $G$ is torsion but it is neither locally finite nor locally solvable.
\end{lemma}
\begin{proof}
By Lemma \ref{Prod+Steph:lemma}, $G$ has no infinite abelian subgroups, so the center of $G$ is finite.

In particular $G$ is torsion, but it is not locally finite by Fact \ref{fac:lfia}. Finally, $G$ is not locally solvable by Fact \ref{fac:lslf}.
\end{proof}

So a discrete HM group is torsion by Lemma \ref{lem:tnlf}. This result should be compared with Proposition \ref{prop:ndhmlc}(2), where we prove that non-discrete locally compact HM groups are not torsion, as they contain a copy of the $p$-adic integers $\Z_p$ for some $p$.

The minimality of infinite discrete groups has been a long standing open question of Markov. Since  minimal discrete groups admit no non-discrete Hausdorff group topologies at all, such groups are also called {\em non-topologizable}.
The first example of a non-topologizable group was provided by Shelah \cite{Sh} under the assumption of the Continuum Hypothesis CH. His example is simple and torsionfree, so that discrete group is also totally minimal, yet not hereditarily minimal. A countable example of a non-topologizable group was built a bit later by
 Ol$'$shanskij  \cite{O} (it was an appropriate quotient of Adjan's group $A(m,n)$ built for the solution of Burnside problem).
	
	There exist infinite \HM\ discrete groups and this fact, recently established by \cite{KOO} is related to another interesting topic,
	namely  {\em categorically compact groups}.  According to  \cite{DU} a Hausdorff topological group $G$ is {\em categorically compact} (briefly, $C$-compact) if for every topological group $H$ the projection $p: G \times H \to H$ sends closed subgroup of $G \times H $ to closed subgroups of $H$.
	Compact groups are obviously $C$-compact by Kuratowski closed projection theorem,
	while solvable $C$-compact groups can be shown to be compact \cite[Corollary 3.5]{DU}. 	
	The  $C$-compact groups are two-sided complete
	and the class of $C$-compact groups has a number of nice properties typical for the compact groups: stability under taking closed subgroups, finite products and
	Hausdorff quotient groups. Moreover, the $\omega$-narrow (in particular, separable) $C$-compact groups are totally minimal \cite[Corollary 3.5]{DU}.
	The question of whether all $C$-compact groups are compact, raised in \cite{DU}, remained open for some time even in the case of locally compact groups
	(the connected $C$-compact locally compact groups are compact \cite[ Proposition 5.1]{DU}). 	
	It was proved in \cite[Theorem 5.5]{DU}, that a countable discrete group $G$ is $C$-compact if and only if every subgroup of $G$ is totally minimal
	(i.e., non-topologizable along with all its quotients).
	This gives rise to the following notion (the specific term was proposed later in \cite{LG2}).
	
\begin{defi}
A group $G$ is {\it hereditarily non-topologizable} when all subgroups of $G$ are non-topologizable along with all their quotients (i.e., totally minimal in the discrete topology).
\end{defi}

It was shown in \cite[Corollary 5.4]{DU} that  every discrete hereditarily non-topologizable group is $C$-compact, yet the existence of infinite hereditarily non-topologizable groups remained open until the recent paper \cite{KOO} which provided many examples of such groups (hence, discrete $C$-compact that obviously fail to be compact). These examples  have various additional properties displaying various levels of being infinite (infinite exponent, non-finitely generated, uncountable, etc.).

\subsection{Proof of Theorem B}\label{subsection:proofB}

The first item of the next result is taken from \cite[Lemma 2.8]{DHXX}.
\begin{lemma}\label{lem:HLM}
	Let $G$ be a topological group, and $N$ be an open subgroup of $G$.
	\ben
	\item
	If $N$ is locally minimal, then $G$ is locally minimal.
	\item If $N$ is \HLM, then $G$  is \HLM.\een
\end{lemma}
\begin{proof}
	$(2)$: Let $H$ be a  subgroup of  $G.$ By our assumption on $N$, the locally minimal subgroup $H\cap N$ is an open in $H$. It follows from item (1) that $H$ is locally minimal.
\end{proof}

The following two deep results are due to Zelmanov and Kaplansky, respectively. Note that  the second part of (a) makes use of Fact \ref{fac:lfia}.
\begin{fact}
\begin{itemize}
\item [(a)]\label{fac:zel} \cite[Theorem 2]{Z}
Every
compact torsion group is locally finite. In particular, every infinite compact group contains an infinite compact abelian subgroup.
\item [(b)]\label{fac:kap} \cite[Theorem 6]{K}
Every non-discrete Lie group contains a non-trivial continuous homomorphic image of $\R$.
\end{itemize}
\end{fact}

The next result is the starting point of our exploration of \HM\ locally compact groups.
Recall that a topological group $G$ is called {\it compactly covered} if each element of $G$ is contained in some compact subgroup.

\begin{prop}\label{prop:ndhmlc}
Let $G$ be a hereditarily minimal locally compact group.
\ben
\item Then $G$ is totally disconnected and compactly covered.
\item  If $G$ is non-discrete, then it contains a copy of $\Z_p$ for some prime $p$ and satisfies ($\mathcal{N}_{fn}$). In particular, $Z(G)$ is trivial or $Z(G) \cong \Z_p$ for some prime $p$.
\een
\end{prop}

\begin{proof} $(1)$ : Assume that the connected component $c(G)$ is non-trivial. As $c(G)$ is a hereditarily minimal connected locally compact group, it is a Lie group by Fact \ref{fac:ext}, and of course it is non-discrete. Using Fact \ref{fac:kap}(b), $c(G) $ contains an infinite abelian  Lie group $K$, which is isomorphic to $\Z_p$ by Corollary \ref{cor:spr}, contradicting the fact that $K$ is a Lie group.
The second assertion follows from Lemma \ref{Prod+Steph:lemma}.

$(2)$: By $(1)$, $G$ is a non-discrete totally disconnected locally compact group. By \cite{vD}, it contains an infinite compact open subgroup $H$. According to Fact \ref{fac:zel}(a),  $H$ contains an infinite \HM \ compact abelian group $A.$ By Fact \ref{TeoP}, $A$ is isomorphic to $\Z_p$ for some prime $p.$

 To prove the second assertion we need to check that $G$ contains no finite normal non-trivial subgroups.
Assume for a contradiction that $G$ has a finite non-trivial normal subgroup $F$.
As $G\geq A\cong \Z_p$, we deduce that
 $A \cap F$ is trivial, and the natural
  action by conjugations $\alpha:A\times F\to F$ has an infinite kernel $M =\ker \alpha$. Being a non-trivial closed subgroup of $A\cong \Z_p$, $M$ is also isomorphic to $\Z_p$. It follows that $F\rtimes_{\alpha}M=F\times M$ is \HM. Let $C$ be a cyclic non-trivial subgroup of $F$. Then $C\times \Z_p$ is a \HM \ compact abelian group, contradicting Fact \ref{TeoP}.
	 For the final assertion, apply Lemma \ref{Prod+Steph:lemma} to the closed normal subgroup $Z(G)$.
\end{proof}

The following example shows that the hypotheses of the above proposition cannot be relaxed.

\begin{example}\label{ex:tarski}
Our first item shows that hereditary minimality cannot be relaxed to hereditary local minimality, while in item (b) we show that an infinite discrete HM group may have finite non-trivial center, so need not satisfy ($\mathcal{N}_{fn}$). So ``non-discrete'' cannot be removed in Proposition \ref{prop:ndhmlc}(2).
\ben[(a)]
\item The Lie group $\R\times \Z(2)$ is certainly hereditarily locally minimal by Fact \ref{fac:ext}, but not totally-disconnected.

\item Let $T$ be a countable discrete hereditarily non-topologizable group (see \S\ref{discreteHM}). By \cite[Theorem 5.5]{DU}, $T$ is $C$-compact. As this class is stable under taking finite products, also the direct product $G = \Z(2) \times T$ is $C$-compact. Since $G$ is also countable and discrete,  \cite[Theorem 5.5]{DU} applies again, and $G$ is hereditarily non-topologizable. Thus, $G$ is a discrete hereditarily minimal group with finite non-trivial center.\een
\end{example}

Proposition \ref{prop:ndhmlc}(2) has the following easy consequence: as $\Z_p \times \Z_q$ is not HM for any pair $p,q$ of primes, direct products of non-discrete locally compact \HM \   groups are never \HM.

Recall that a profinite group is a totally disconnected compact group. The following  proposition can be applied for example to profinite torsionfree groups, as a torsionfree group obviously satisfies ($\mathcal{N}_{fn}$).

\begin{prop}\label{pro:tor}
Let $G$ be an infinite profinite group satisfying ($\mathcal{N}_{fn}$). If $H$ is a locally minimal dense subgroup of $G$, then $H$ is minimal.
\end{prop}

\begin{proof}
We prove first that if $H$ is a locally essential subgroup of $G$, then it is  essential in $G$.
	Since $G$ is profinite, there exists a local base  at the identity $\mathcal B$ consisting of compact open normal subgroups. Suppose that $H$ is locally essential in $G$, and let $K\in \mathcal B$ be such that  $M\cap H$ is not trivial for every non-trivial closed normal subgroup $M$ of $G$ contained in $K$.
	
	Let $N$ be a non-trivial closed normal subgroup of $G$, and we will show that $N\cap H$ is not trivial.
	As $N$ is compact, and $M=N\cap K$ is an open subgroup of $N$, the index $[N:M]$ is finite.
	This yields that $M$ is infinite since $N$ is infinite by our assumption on $G$.
	Hence, $M\subseteq K$ is a non-trivial closed normal subgroup of $G.$ This implies that $M\cap H$ is not trivial. Since $M\cap H\subseteq N\cap H$ the latter group is also non-trivial.

 Let $H$ be a locally minimal dense subgroup of
	the compact group $G$. By the  Local Minimality Criterion, $H$ is locally essential in $G$, hence essential by the above argument. Now apply the Minimality Criterion to conclude that $H$ is minimal.
\end{proof}

The compact group $\T$ has plenty of dense non-minimal subgroups and they are all locally minimal by Fact \ref{fac:ext}, so ``profinite'' cannot be relaxed to ``compact'' in the above proposition.

By Proposition \ref{prop:ndhmlc}, the compact HM groups satisfy the assumptions of Proposition \ref{pro:tor}. Using this fact, now we prove Theorem B: a compact group is \HM \ if and only if it is \HLM \ and satisfies ($\mathcal{C}_{fn}$).

\bigskip

\noindent{\bf Proof of Theorem B.}
If $G$ is a \HM \ group, then  it is \HLM. To prove $G$ satisfies ($\mathcal{C}_{fn}$), pick an infinite compact subgroup $H$ of $G$. Clearly, $H$ is also \HM, so
satisfies ($\mathcal{N}_{fn}$) by Proposition \ref{prop:ndhmlc}(2). This means that $G$ satisfies ($\mathcal{C}_{fn}$).

For the converse implication, suppose that $G$ is a \HLM\ compact group satisfying ($\mathcal{C}_{fn}$).

We first show that $G$ is totally disconnected. Assuming the contrary, the connected component $c(G)$ is a non-trivial hereditarily locally minimal, connected, compact group. By Fact \ref{fac:ext}, $c(G)$ is a non-discrete Lie group.
According to Fact \ref{fac:kap}(b), $c(G)$ contains an infinite compact abelian Lie group  $C$. In particular, $C$ contains a copy of $\T$, hence contains torsion elements. So $C$ does not satisfy ($\mathcal{N}_{fn}$), contradicting  our assumption that $G$ satisfies ($\mathcal{C}_{fn}$).

Let $H$ be a subgroup of the profinite group $G$. Without loss of generality we may assume that $H$ is infinite. By our assumption, $\overline H$ is a profinite \HLM\ group satisfying ($\mathcal{N}_{fn}$). By Proposition \ref{pro:tor}, applied to $\overline H$ and its subgroup $H$, we deduce that $H$ is minimal.\qed

\bigskip

Our first application of Theorem B
shows that the groups $M_{p,n}$ are $\hm$ for every prime $p$ and every $n\in \N.$
\begin{example}\label{padics:rtimes:padics}
Being a subgroup of $(\Z_p,+)\rtimes \Z_p^*$, the group $M_{p,n}=(\Z_p,+)\rtimes C_p^{p^n}$ is \HLM\ by Theorem A. As $M_{p,n}$ is torsionfree, it is also \HM\ by  Theorem B.
\end{example}

The following proposition proves the implication (2) in the diagram in the introduction under an assumption much weaker than total disconnectedness (equivalent to the property
of having a compact connected component\footnote{or containing no lines, i.e., copies of $\R$.}).

\begin{prop} \label{prop:ltdc}
	Let $G$ be a locally compact group having a compact open subgroup $H$.
	\ben
	\item If $G$ is \CHLM, then it is \HLM.
	\item If $G$ is \CHM, then it is either discrete or contains a copy of $\Z_p$ for some prime $p$.
	\een
\end{prop}

\begin{proof}

$(1)$:  Since  $G$ is \CHLM, $H$ is \HLM, and Lemma \ref{lem:HLM}(2) applies.

$(2)$: If $G$ is non-discrete and \CHM,  then also $H$ is non-discrete. Being also \HM\ and compact, Proposition \ref{prop:ndhmlc}(2) applies to $H$.
\end{proof}

\section{Proof of Theorem C}\label{lcHM}

We start this section with  the following apparently folklore property of central extensions. Recall that a group is \emph{$p$-torsionfree} if it has no element of order $p$.
\begin{lemma}\label{lem:Z1Z2}
Let $G$ be a group with non-trivial center. If $Z(G)$ is $p$-torsionfree and $G/Z(G)$ is a $p$-group, then $Z_2(G)=Z(G)$. If in addition $G/Z(G)$ is finite, then $G$ is abelian.
\end{lemma}
\begin{proof}
Let $Z_1=Z(G)$ be $p$-torsionfree and $G/Z_1$ be a $p$-group. Assume for a contradiction that  $Z_1\lneqq Z_2$, where $Z_2=Z_2(G)$. We can choose $x\in Z_2 \setminus Z_1$  such that $x^p \in  Z_1$. Since $x\notin Z_1$, there exists $y\in G$ that does not commute with $x$. Then $a =[x,y]=xyx^{-1}y^{-1} \ne e$, and $a\in Z_1$, by the definition of $Z_2$. Let $\phi:G\to G$ be the inner automorphism induced by $x$ and observe that $\phi(y)=xyx^{-1}=ay$. Since $x^p \in Z_1$ is a central element, $\phi^p = Id_G.$ On the other hand, as $a\in Z_1$, we also have $\phi^p(y) = {x^p}yx^{-p}=a^py.$ So $y = a^py$, and it follows that $a^p = e$. Since $a\in Z_1$, which is $p$-torsionfree, this yields $a = e$, a contradiction.

For the last assertion, recall that finite $p$-groups are nilpotent, so  if $G$ is not abelian, then $Z(G/Z_1(G))$ is non-trivial, i.e., $Z_1(G)\lneqq Z_2(G)$.
\end{proof}

In the following result, we prove some more properties a non-discrete \HM, locally compact  group with non-trivial center satisfies.
\begin{corol}\label{cor:pgr}
	Let $G$ be a non-discrete \HM \ locally compact  group with non-trivial center. Then there is a prime number $p$ such that:
	\ben
	\item for every non-central element $x \in G$, we have $\Z_p \cong Z(G) \leq \overline{\langle x\rangle} \cong \Z_p$. In particular, for every closed subgroup $H$ of $G$, either $H \leq Z(G)$, or $Z(G) \leq H$;

	\item $G$ is torsionfree, and $G/Z(G)$ is a center-free $p$-group.
	\een
\end{corol}
\begin{proof}
	$(1)$: The center $Z(G)$ is isomorphic to $\Z_p$ for some prime $p$ by Proposition \ref{prop:ndhmlc}(2).
	
	Let $x \in G $ be a non-central element, and consider the abelian subgroup $N_x = \overline{\langle Z(G), x\rangle}$ of $G$. As $N_x$ contains $Z(G) \cong \Z_p$, it is isomorphic to $\Z_p$ by Lemma \ref{Prod+Steph:lemma}. As the closed subgroups of $\Z_p$ are totally ordered by inclusion, we obtain that $Z(G) \leq \overline{\langle x\rangle} = N_x \cong \Z_p$.
	It follows that if $H$ is a closed non-central subgroup of $G$, then $Z(G)\leq H$.  $N_x$.
	
	$(2)$: If $x\notin Z(G)$, then by item $(1)$, the index $[\overline{\langle x\rangle}:Z(G)]=p^n$ for some $n\in \N$, so $x^{p^n}\in Z(G)$. This proves that $G/Z(G)$ is a $p$-group and that $G$ is torsionfree, so $Z_2(G)=Z(G)$ by Lemma \ref{lem:Z1Z2}.
\end{proof}

In the next result, we use Corollary \ref{cor:pgr} to extend Corollary \ref{cor:spr}.
\begin{corol}\label{cor:nprodanov}
An infinite \HM\ locally compact hypercentral group $G$ is isomorphic to $\Z_p$ for some prime $p$.
\end{corol}
\begin{proof} By a well-known theorem of Mal'cev (see \cite[page 8]{DIX}), all hypercentral groups are locally nilpotent, so in particular $G$ is locally solvable. Then $G$ is not discrete by Lemma \ref{lem:tnlf}, and clearly it has  non-trivial center. Hence $Z_2(G)=Z(G)\cong \Z_p$ by Corollary \ref{cor:pgr}. Since $G$ is hypercentral, we deduce that $G=Z(G)\cong \Z_p$.
\end{proof}

The next results will be used in the subsequent proof of Theorem C.
\begin{prop}\label{new:thmB}
If $G$ is a \HM \ locally compact group with non-trivial center, then every abelian subgroup of the quotient group $G/Z(G)$ is finite. 
\end{prop}
\begin{proof}
Assume that $G/Z(G)$ has an abelian subgroup, so  also its closure $A$ is abelian. Let $H = q^{-1}(A)$, where  $q:G \to G/Z(G)$ is the canonical homomorphism. So, $H$ is a locally compact subgroup of $G$ with $H \geq Z(G)$. Hence, $Z(H) \geq Z(G)\cap H=Z(G)$, and by the third isomorphism theorem we deduce that $H/Z(H)$ is a quotient of the abelian group $H/Z(G) \cong A$, so $H$ is nilpotent (of class $\leq 2$).

 If $G$ is discrete, then $H$ is finite by Lemma \ref{lem:tnlf}. As $H = q^{-1}(A)$, we deduce that also $A$ is finite.
Now we assume that $G$ is non-discrete.  By Proposition \ref{prop:ndhmlc}(2), $Z(G) \cong \Z_p$ for some prime $p$. As $H$ contains $Z(G)$,   Corollary \ref{cor:nprodanov} implies that also $H$ is isomorphic to $\Z_p$. So $A \cong H/Z(G)$ is finite.
\end{proof}

The quotient group $Q=G/Z(G)$ is torsion, by Lemma \ref{lem:tnlf} and Corollary \ref{cor:pgr}.
According to Fact \ref{fac:lslf} and Fact \ref{fac:lfia}, the property of $Q$ described in Proposition \ref{new:thmB} is equivalent also to having no infinite locally solvable subgroup, or having no infinite locally finite subgroup.

\begin{corol}\label{new:corol:for:ThmB}
Let $G$ be a non-discrete \HM \ locally compact  group with non-trivial center. If $G/Z(G)$ is a locally finite group, then $G \cong \Z_p$ for some prime $p$.
\end{corol}
\begin{proof}
By Corollary \ref{cor:pgr}, we obtain that $G/Z(G)$ is a $p$-group and $Z(G)\cong \Z_p$.

If $G/Z(G)$ is infinite, then it has an infinite abelian subgroup by Fact \ref{fac:lfia}, contradicting Proposition \ref{new:thmB}. So $G/Z(G)$ is finite, and $G$ is abelian by Lemma \ref{lem:Z1Z2}.
\end{proof}

As a consequence of Corollary \ref{new:corol:for:ThmB} we now prove Theorem C.

\medskip

\noindent
{\bf Proof of Theorem C.} Let $G$ be an infinite \HM\  locally compact group with non-trivial center that is either compact or locally solvable. We have to prove that $G \cong \Z_p$, for some prime $p$.

First note that $G$ is non-discrete by Lemma \ref{lem:tnlf}. Applying Corollary \ref{cor:pgr}, we obtain that $G/Z(G)$ is a $p$-group and $Z(G)\cong \Z_p$. In view of Corollary \ref{new:corol:for:ThmB}, it suffices to prove that $G/Z(G)$ is locally finite.
If $G$ is locally solvable, then its quotient $G/Z(G)$ is locally finite by Fact \ref{fac:lslf}.
If $G$ is compact, then $G/Z(G)$ is compact torsion, so locally finite by Fact \ref{fac:zel}(a).
\qed

\bigskip

Let us see that the  assumption ``compact or locally solvable'' cannot be removed in Theorem C.
Recall that the countable discrete group $G = \Z(2) \times T$ from Example \ref{ex:tarski} is hereditarily minimal with non-trivial center. Clearly, this group is non-abelian, and it is neither locally solvable nor compact.
	
\medskip

Applying Theorem C to the groups studied in Corollary \ref{cor:pgr} (see also Proposition \ref{new:thmB}), we can deduce additional properties they share.

\begin{thm}\label{add:prop:thm}
	Let $G$ be a non-discrete hereditarily minimal locally compact group, and assume $\{e\} \neq Z(G) \neq G$.
	Then there exists a prime $p$ such that:
	\ben
	
	\item every non-trivial closed subgroup $H$ of $G$ (e.g., $Z(G)$) is open;
	\item every non-trivial compact subgroup $H$ of $G$ is isomorphic to $\Z_p$;
	\item every finite subgroup of $G/Z(G)$ is cyclic and  $G/Z(G)$  satisfies $(\mathcal N_{fn})$.
	\een
\end{thm}

\begin{proof}
	By Corollary \ref{cor:pgr}, $G$ is torsionfree, $Z(G)\cong \Z_p$ and $G/Z(G)$ is a $p$-group, for some prime $p$.
	
	$(1)$: First we prove that $Z(G)$ is open in $G$.
	By Proposition \ref{prop:ndhmlc}(1), $G$ is totally disconnected, so it has a local base at the identity consisting of compact open subgroups, and let $K$ be one of these subgroups. If $K$ is central, then $Z(G)$ is open. Otherwise, we have $K\geq Z(G)$ by Corollary \ref{cor:pgr}(1), so $K\cong \Z_p$ by Theorem C.  Moreover, the index $[K:Z(G)]$ is finite and $Z(G)$ is open in $K$, so $Z(G)$ is open in $G$.

	Now let $H$ be a non-trivial closed subgroup of $G$. If it is contained in $Z(G)$ then, being closed in $Z(G) \cong \Z_p$, it is also open in $Z(G)$, hence open in $G$. Otherwise, $H$ contains $Z(G)$ by Corollary \ref{cor:pgr}(1), so $H$ is open.
	
	$(2)$: If $H$ is a non-trivial compact subgroup of $G$, then $H$ is infinite as $G$ is torsionfree. This implies that $H\cap Z(G)\neq \{e\}$, since $Z(G)$ is open in $G$. So we deduce that $H \cong \Z_p$ by Theorem C.

	$(3)$: Let $\pi: G \to G/Z(G)$ be the canonical map, and let $F$ be a finite subgroup of $G/Z(G)$.
	Then $\pi^{-1}(F)$ is a closed subgroup of $G$, containing $\ker \pi = Z(G)\cong \Z_p$, and such that $[\pi^{-1}(F) : Z(G)] = |F|$ is finite. Then $\pi^{-1}(F)$ is compact, hence isomorphic to $\Z_p$ by item $(2)$. If $g \in G$ is such that $\Z_p \cong \overline{\langle g \rangle} = \pi^{-1}(F)$, then $F = \langle \pi (g) \rangle$.
	
	To check that $G/Z(G)$  satisfies $(\mathcal N_{fn})$ pick a finite non-trivial normal subgroup $N$  of $G/Z(G)$ and let $|N| = p^n$.
	By what we have just proved, $H = \pi^{-1}(N)$ is isomorphic to $\Z_p$, and indeed $\Z_p \cong Z(G) = p^n H$. Let $x \in H \setminus Z(G)$, and $y \in G$ be an element non-commuting with $x$. As $H$ is normal in $G$, we can consider the conjugation by $y$ as a map $\phi : H \to H$. Since $p^n x \in Z(G)$, we have $\phi (p^n x ) = p^n x$; on the other hand, $\phi (p^n x ) = p^n \phi ( x )$, so we conclude $p^n \phi ( x ) = p^n x$. As $H$ is abelian and torsionfree, we deduce $\phi(x) = x$, a contradiction.
\end{proof}

Note that $G$ as in Theorem \ref{add:prop:thm} is neither compact nor locally solvable, by Theorem C.
See Questions \ref{que:exist} and \ref{QQ:DC} for further comments and Proposition \ref{prop:cand} for a partial converse.

We conclude this section by listing the three possibilities (trichotomy) for an infinite  locally compact non-abelian HM group $G$:

\begin{enumerate}
\item $G$ is discrete, if and only if $G$ is torsion (by Lemma \ref{lem:tnlf} and Proposition \ref{prop:ndhmlc}(2)). In this case $G$ is not locally finite but may have finite non-trivial center (see Example \ref{ex:tarski}).

\hspace{-1.0cm}If $G$ is not discrete, we apply Proposition \ref{prop:ndhmlc}, and we obtain the following two cases:
\item $Z(G) = \{e\}$, $G$ contains a copy of $\Z_p$, and satisfies ($\mathcal{N}_{fn}$).
\item $Z(G) = Z_2(G) \cong \Z_p$ is open and proper in $G$, and $G$ has the properties listed in Corollary \ref{cor:pgr} and Theorem \ref{add:prop:thm}.
\end{enumerate}

To case (1) was dedicated \S 3.1, while case (2) (for solvable groups) will be the subject of the rest of the paper.

Note that the groups in (1) and in (3) are neither compact, nor locally solvable, by Lemma \ref{lem:tnlf} and Theorem C.

\section{Semidirect products of $p$-adic integers}\label{Semidirect products of p-adic integers}

We begin this section with a general result on semidirect products, and their quotients. We then apply it in the subsequent Lemmas \ref{lemma:speder} and \ref{commTn}, where we consider some of the semidirect products introduced in Example \ref{Exaaa}.
\begin{lemma}\label{claim:gender}
	Let $\alpha: Y\times X\to X$ be a continuous action of an abelian group $Y$ on an abelian group $X$, and consider the topological semidirect product $G = X\rtimes_\alpha Y$.
	Consider the subgroup of $X$ defined by $A = \langle x-\alpha(y,x): x\in X, y\in Y \rangle$.
	
	Then the derived group $G'$ coincides with $A \rtimes_\alpha \{e_Y\}$, and the quotient group $G/G'$  is  topologically isomorphic to  $(X/A) \times Y$.
\end{lemma}
\begin{proof}
	For $y\in Y$, let $A_y = \{x-\alpha(y,x): x\in X\}$, then  $A = \langle A_y: y \in Y\rangle $. Note that the commutator $[(x,e_X), (e_Y,y)] = (x-\alpha(y,x),e_Y)$ for every $x\in X$ and $y\in Y$, so
	$A \rtimes \{e_Y\} \leq G'$.
	
	On the other hand, every $y,y'\in Y$ commute, thus
	\[
	\alpha(y', A_y)  = \{ \alpha(y',x) - \alpha(y', \alpha(y,x)) : x\in X\} = \{ \alpha(y',x) - \alpha(y, \alpha(y',x)) : x\in X\} \leq A_y \leq A,
	\]
	which implies that  $A \rtimes_\alpha \{e_Y\}$ is normal in $G$.
Let $\chi:G\to (X/A) \times Y$ be defined by $\chi(x,y)=(\phi(x),y)$, where $\phi:X\to X/A$ is the canonical map. Since $\phi$ is a continuous open surjection it follows that $\chi$ is a continuous open surjection. Moreover, the definitions of $\phi$ and $A$    imply that  $\chi$ is also homomorphism. It is easy to see that $\ker\chi=A \rtimes_\alpha \{e_Y\}$, so  $G \big/(A \rtimes_\alpha \{e_Y\})\cong(X/A) \times Y$ is abelian and $G'\leq A \rtimes_\alpha \{e_Y\}$.
\end{proof}

Recall that the topological group $(C_p,\cdot)$ is isomorphic to the group $(\Z_p,+)$ (essentially, via the $p$-adic logarithm), so  the closed subgroups of $(C_p,\cdot)$ are totally ordered, have the form $C_p^{p^n}$, for  $n \in \N$, and have been described in (\ref{Cppn:form}).

Now we apply  Lemma \ref{claim:gender} to the case when $(C_p^{p^n},\cdot)$, viewed as a subgroup of the automorphism group $\Aut(\Z_p)$, acts  on $\Z_p$ via the natural  action by multiplication. Recall that we denote by $M_{p,n}= \Z_p \rtimes C_p^{p^n}$ the semidirect product arising this way.
\begin{lemma}\label{lemma:speder}

For the group $M_{p,n}= \Z_p \rtimes C_p^{p^n}$, the following hold:
\begin{itemize}
\item if $p>2$, then $M_{p,n}' = p^{n+1}\Z_p \rtimes \{1\}$ and $M_{p,n} / M_{p,n}' \cong \Z(p^{n+1})\times C_p^{p^n}$;
\item if $p=2$, then $M_{2,n}' = 2^{n+2}\Z_2 \rtimes \{1\}$ and $M_{2,n} / M_{2,n}' \cong \Z(2^{n+2})\times C_2^{2^n}$.
\end{itemize}
\end{lemma}
\begin{proof}
In the notation of Lemma \ref{claim:gender}, let $A= \langle (1-y)x: x\in \Z_p , y\in C_p^{p^n}\rangle$.

By (\ref{Cppn:form}), one immediately obtains that if $p>2$, then $A = p^{n+1}\Z_p$, while $A = 2^{n+2}\Z_2$ otherwise.
\end{proof}

In the following remark we give the explicit isomorphisms stated in the above Lemma \ref{lemma:speder}. 
\begin{remark}\label{rem:last}
Using the proof of Lemma \ref{claim:gender} we can write explicitly the isomorphisms in Lemma \ref{lemma:speder}.
Let $p$ be a prime and $n\in \N$. The isomorphism $\widetilde{\psi}$ satisfies the equality  $\widetilde{\psi}((x,y)M_{p,n}')=\psi(x,y) $ for every $(x,y)\in M_{p,n}$, where:
\begin{itemize}
	\item if $p>2$, $\psi:M_{p,n}\to \Z(p^{n+1})\times C_p^{p^n}$ is defined by $\psi(x,y)=(x\mod p^{n+1},y )$;
	\item if $p=2$, $\psi:M_{2,n}\to \Z(2^{n+2})\times C_2^{2^n}$ is defined by $\psi(x,y)=(x\mod 2^{n+2},y )$.
\end{itemize}
In other words, we have $\psi=\widetilde{\psi}\circ q$, where $q:M_{p,n}\to M_{p,n}/M_{p,n}'$ is the canonical map.
\end{remark}

Obviously, if two groups $M_{p,n}$, $M_{p',n'}$ are isomorphic, then $p = p'$. Under this assumption, we now prove that also $n = n'$.
\begin{corol}\label{thm:pair}
For $n \in \N$, the subgroups $M_{p,n}= \Z_p\rtimes C_p^{p^n}$ of $M_{p,0}=M_p= \Z_p\rtimes C_p$ are pairwise non-isomorphic.
\end{corol}
\begin{proof}
Let $n,m \in \N$, and assume that $\psi: M_{p,n} \to M_{p,m}$ is an isomorphism. Then $\psi(M_{p,n}') = M_{p,m}'$, and
$\psi$ induces an isomorphism $\bar \psi: A \to B$, where $A = M_{p,n}/M_{p,n}'$, and $B = M_{p,m}/M_{p,m}'$.

By Lemma \ref{lemma:speder}, comparing the torsion subgroups of $A$ and $B$, we obtain $\Z(p^{n+1}) \cong \Z(p^{m+1})$ when $p>2$, or  $\Z(2^{n+2})  \cong \Z(2^{m+2})$ when $p=2$. In any case, $n = m$.
\end{proof}

The following result is the counterpart of Lemma \ref{lemma:speder} for the groups $T_n$.
\begin{lemma}\label{commTn}
For every $n\in \N$, the commutator subgroup of the group $T_n$ is $T_n'= 2\Z_2\rtimes\{1\}$, and the quotient group $T_n / T_n'$ is isomorphic to $\Z(2) \times C_2^{2^n}$.
\end{lemma}
\begin{proof}

Let $A=\langle x-\beta(y,x): x\in \Z_2, \ y\in C_2^{2^{n}} \rangle \leq \Z_2$.
In view of Lemma \ref{claim:gender}, we have to prove that $A=2\Z_2$. By the definition of $\beta$ we have  $A=\langle V, W\rangle$, where $V=\langle x-yx: x\in \Z_2, \ y\in C_2^{2^{n+1}} \rangle \leq \Z_2$ and $W=\langle  x+yx: x\in \Z_2, \ y\in C_2^{2^{n}}\setminus C_2^{2^{n+1}}\rangle \leq \Z_2$.

Note that if $y \in C_2^{2^{n+1}}$, then $y \in 1 + 2\Z_2$, so $1-y \in 2 \Z_2$. In particular, $V=\langle (1-y)x: x\in \Z_2, \ y\in C_2^{2^{n+1}} \rangle \leq 2\Z_2$.

To study $W$, first observe that $C_2^{2^{n}}\setminus C_2^{2^{n+1}}=(1+2^{n+2}\Z_2)\setminus (1+2^{n+3}\Z_2)=1+2^{n+2}+2^{n+3}\Z_2$. Hence,
\[
W=\langle x(1+y):  x\in \Z_2, \ y\in 1+2^{n+2}+2^{n+3}\Z_2 \rangle =\langle xt:x\in \Z_2, \ t\in 2+2^{n+2}+2^{n+3}\Z_2\rangle \leq 2\Z_2.
\]

On the other hand,
$2\Z_2=(2+2^{n+2})\Z_2\leq W$, since $t=2+2^{n+2}\in 2+2^{n+2}+2^{n+3}\Z_2$.
It is now clear that $A=\langle V, W\rangle= W= 2\Z_2$, which completes the proof.
\end{proof}

\begin{prop}\label{prop:p=2}
For $n\in \N$, the groups $T_n$ are pairwise non-isomorphic.
\end{prop}
\begin{proof} Assume that there exists a topological isomorphism $\psi: T_n\to T_m$. Let $\pi_2:T_m\to C_2^{2^m}$ be the projection  on the second coordinate. We first show that $\pi_2(\psi(\Z_2\rtimes\{1\}))=1$. If $a\in \Z_2$ and $\pi_2(\psi(a,1))=c$, then $\pi_2(\psi(2a,1))=c^2$. By Lemma \ref{commTn}, $T_n'=T_m' = 2\Z_2\rtimes\{1\}$, which implies that $\psi(2\Z_2\rtimes\{1\})=2\Z_2\rtimes\{1\}$  and $\pi_2(\psi(2\Z_2\rtimes\{1\}))=1$. As $(2a,1)\in 2\Z_2\rtimes\{1\}$, we deduce that $c^2=1$. It follows that $c=1$, since $C_2^{2^m}$ is torsionfree.

Consider the subgroups $M_{2,n+1}=\Z_2\rtimes_{\beta_n}C_2^{2^{n+1}}$ and $M_{2,m+1}=\Z_2\rtimes_{\beta_m}C_2^{2^{m+1}}$ of $T_n$ and $T_m$, respectively. We will prove that $\psi(M_{2,n+1})=M_{2,m+1}$. Since $\psi^{-1}$ is also a topological isomorphism, it suffices to show that $\psi(M_{2,n+1})\leq M_{2,m+1}$.

For this aim we prove that $\pi_2(\psi(M_{2,n+1}))\leq C_2^{2^{m+1}}$. Note that an element of $M_{2,n+1}$ has the form $(a,b^2)$, where $a\in \Z_2$ and $b\in C_2^{2^n}$. Clearly, $\pi_2(\psi(0,b))\in C_2^{2^m}$, so $$\pi_2(\psi(a,b^2))=\pi_2(\psi(a,1))\pi_2(\psi(0,b))^2= 1\cdot \pi_2(\psi(0,b))^2\in C_2^{2^{m+1}}.$$  Hence, $\pi_2(\psi(M_{2,n+1}))\leq C_2^{2^{m+1}}$, which means that $\psi(M_{2,n+1})\leq M_{2,m+1}$. By Corollary \ref{thm:pair}, we deduce that $m=n$.
\end{proof}

It is not difficult to see that if $T_n \cong M_{p,m}$ for some $n,m,p$, then $p=2$. In the next result we apply Lemma \ref{commTn} to prove also that a group $T_n$ is not isomorphic to any of the groups $M_{2,m}$.
\begin{prop}\label{Tn-Mpn:non-isom}
For every $n,m\in \N$, and every prime number $p$, the groups $T_n$ and $M_{p,m}$ are not isomorphic.
\end{prop}
\begin{proof}
By contradiction, assume that $\psi: T_n \to M_{p,m}$ is an isomorphism. Similar to the proof of Corollary \ref{thm:pair}, $\psi$ induces an isomorphism between the torsion groups $t(T_n/T_{n}')$ and $t(M_{p,n}/M_{p,n}')$.
But $t(T_n/T_{n}')\cong \Z(2)$ by Lemma \ref{commTn}, while $t(M_{2,n}/M_{2,n}')\cong \Z(2^{n+2})$ and $t(M_{p,n}/M_{p,n}')\cong \Z(p^{n+1})$ when $p>2$ by Lemma \ref{lemma:speder}.
\end{proof}

In the sequel, we consider a faithful action $\alpha: \Z_p \times \Z_p \to \Z_p$ of $\Z_p$ on $\Z_p$, and the semidirect product $M_{p,\alpha}= (\Z_p,+)\rtimes_{\alpha} (\Z_p,+)$ arising this way.
Recall that $\Aut(\Z_p) \cong \Z_p^*$, as every $\phi\in \Aut(\Z_p)$, has the form $\phi(x)=m\cdot x$ for $m =\phi(1) \in \Z_p^*$; identifying $\phi$ with $\phi(1)$, the action $\alpha$ gives a group monomorphism $f:(\Z_p,+)\to \Z_p^*$ such that $\alpha(y,x)= f(y)\cdot x$.
\begin{prop}\label{prop:alpha}
For a prime $p$, consider the semidirect product $M_{p,\alpha}= (\Z_p,+)\rtimes_{\alpha} (\Z_p,+)$, where $\alpha$ is a faithful action.
\begin{itemize}
\item If $p >2$, then $M_{p,\alpha} \cong M_{p,n}$ for some $n\in \N$.
\item If $p =2$, then either $M_{p,\alpha} \cong M_{p,n}$, or $M_{p,\alpha} \cong T_n$, for some $n\in \N$.
\end{itemize}
\end{prop}
\begin{proof}
Since $\alpha$ is faithful, there is a group monomorphism $f:(\Z_p,+)\to \Z_p^*= C_pF_p$ such that $\alpha(y,x)= f(y)\cdot x$. Now we consider two cases, depending on whether the image of $f$ is contained in $C_p $ or not.

If it is, then $f:(\Z_p,+)\to C_p$   is continuous when we equip these two copies of $(\Z_p,+)$ with the $p$-adic topology. So $f(\Z_p)= C_p^{p^n}$ for some $n\in \N$, and $f: (\Z_p,+)\to C_p^{p^n}$ is a topological isomorphism.

We define $\phi:(\Z_p,+)\rtimes_{\alpha} (\Z_p,+)\to (\Z_p,+)\rtimes (C_p^{p^n},\cdot)$ by $(x,y)\to (x,f(y))$. To prove $\phi$ to be a topological isomorphism, it remains only to check it is a homomorphism, as follows:
\begin{equation}\label{homomorphism}
\begin{split}
\phi((x_1,y_1)(x_2,y_2))=\phi(x_1+\alpha(y_1,x_2), y_1y_2)=\phi(x_1+f(y_1)\cdot x_2, y_1y_2)=(x_1+f(y_1)\cdot x_2, f(y_1)f(y_2)),
\\
\phi(x_1,y_1)\phi(x_2,y_2)=(x_1,f(y_1))(x_2,f(y_2))=(x_1+f(y_1)\cdot x_2, f(y_1)f(y_2)).
\end{split}
\end{equation}

Now we consider the case when  $f(\Z_p) \nsubseteq C_p$.
First we see that this happens only if $p=2$; indeed this follows from the fact that  if $p>2$ then $(\Z_p,+)$ is $(p-1)$-divisible, while $F_p$ has cardinality $p-1$.
So we have  $p=2$ and $f:(\Z_2,+)\to \Z_2^* =C_2 F_2$ such that $\alpha(y,x)= f(y)\cdot x$ and $f(\Z_2) \nsubseteq C_2$.
Recall that  $C_2=1+4\Z_2$ and  $F_2 =\{ 1,-1 \}$.

Equipping $C_2 \cong \Z_2$ with the $2$-adic topology, $F_2$ with the discrete topology, and the codomain of $f$ with the product topology, it is easy to see that $f: (\Z_2,+)\to C_2 \cdot F_2$ is continuous,
so $f(1)\in (-1)\cdot C_2$.

Consider the (continuous) canonical projection $\pi: C_2 \cdot F_2 \to C_2$, and call $\tilde f$ the composition map $\tilde f = \pi \circ f : \Z_2 \to C_2$. Then $\tilde f$ is a continuous homomorphism, $\tilde f (1)=-f(1)$, and it is easy to see that
\begin{equation*}
\tilde f(y) =
\begin{cases}
f(y) & \text{ if } y\in 2\Z_2,\\
-f(y) & \text{ if } y \in \Z_2 \setminus 2\Z_2.
\end{cases}
\end{equation*}
Since $\Z_2$ is compact, there is $n\in \N$ such that $\tilde f(\Z_2)=C_2^{2^n}=(1+4\Z_2)^{2^n}$, and we now prove that $M_{2,\alpha}$ is isomorphic to $T_n$.

Let $\phi:M_{2,\alpha}\to T_{n}$ be defined by $\phi(x,y)=(x,\tilde f(y))$.
As $\tilde f: (\Z_2,+) \to C_2^{2^n}$ is a topological group isomorphism we deduce that $\phi$ is homeomorphism. Let us show that $\phi$ is also a homomorphism. If $(x_1,y_1), (x_2,y_2)\in M_{2,\alpha}$, then
\[
\phi(x_1,y_1) \phi (x_2,y_2)=(x_1,\tilde f(y_1))(x_2,\tilde f(y_2))=
( x_1+\beta_n \big(\tilde f(y_1),x_2 \big) , \tilde f(y_1)\tilde f(y_2) ).
\]
On the other hand, since $\tilde f$ is a homomorphism we have
\[
\phi((x_1,y_1)(x_2,y_2))= \phi(x_1+f(y_1) \cdot x_2,y_1+y_2)=
(x_1+f(y_1) \cdot x_2,\tilde f(y_1+y_2))=(x_1+f(y_1) \cdot x_2,\tilde f(y_1)\tilde f(y_2) ).
\]

To finish the proof, we now check that $\beta_n(\tilde f(y_1),x_2) = f(y_1) \cdot x_2$.

If $y_1\in 2\Z_2$, then $\tilde f(y_1)=f(y_1)$ and also $\tilde f(y_1)\in C_2^{2^{n+1}}$, so
$\beta_n(\tilde f(y_1),x_2)=\tilde f(y_1)\cdot x_2=f(y_1) \cdot x_2$.

If $y_1 \in \Z_2 \setminus 2\Z_2$, then $\tilde f(y_1)=-f(y_1)$, and moreover $\tilde f(y_1)\in C_2^{2^{n}} \setminus C_2^{2^{n+1}}$, so  $\beta_n(\tilde f(y_1),x_2)=-\tilde f(y_1)\cdot x_2=f(y_1)\cdot x_2$.
\end{proof}

In the following lemma we describe the faithful actions of a finite group on the $p$-adic integers.

\begin{lemma}\label{lem:uonique}
Let $\alpha: H\times \Z_p\to \Z_p$ be a faithful action of a finite group $H$ on $\Z_p$. Then $H$ is isomorphic to a subgroup $F$ of $F_p$, and $G=(\Z_p,+)\rtimes_{\alpha} H$ is topologically isomorphic to $K_{p,F}$.
\end{lemma}
\begin{proof}
Since $\alpha$ is faithful, there is a group monomorphism $f:H\to \Z_p^*$ such that $\alpha(h,x)= f(h)\cdot x$. As $H$ is finite, its image $F$ is contained in the torsion subgroup $F_p$ of $\Z_p^*$.

Consider the map $\phi: G\to K_{p,F}$ defined by $\phi(a,b)=(a,f(b)) $ for every $(a,b)\in K_{p,F}$.
Following the argument in (\ref{homomorphism}), one can verify that $\phi$ is a topological isomorphism.
\end{proof}

\section{Proof of Theorem D}\label{Proof of Theorem D}
In this section we prove  that infinite locally compact  solvable $\hm$  groups are metabelian (see Proposition \ref{prop:ms}). Then we  use it  to classify all locally compact solvable $\hm$  groups.

We start this section with two general results we use in the sequel. Especially Lemma \ref{lem:difp} will be used in Theorem \ref{thm:charmeta}, Proposition \ref{prop:tri}, Proposition \ref{prop:dich} and Theorem \ref{prop:nonfree}.

\begin{lemma} \label{lem:meta}
Let $G$ be an 	infinite \HM \ locally compact solvable group of class $n>1$. Then, $\overline{G^{(n-1)}}$ is isomorphic to $\Z_p$ for some prime $p$.
\end{lemma}
\begin{proof}
Since $G$ is solvable  of class $n>1$, $G^{(n-1)}$ is a non-trivial normal abelian subgroup of $G$. By Lemma \ref{lem:tnlf}, $G$ is non-discrete.
 Hence, $G^{(n-1)}$ is infinite, by Proposition \ref{prop:ndhmlc}. Being infinite  \HM\ locally compact abelian group, $\overline{G^{(n-1)}}$ is isomorphic to $\Z_p$ for some prime $p$, by Corollary \ref{cor:spr}.
\end{proof}

Let $p$ be a prime. Recall that an abelian group $G$ is {\em $p$-divisible} if $pG=G$.
\begin{lemma} \label{lem:difp}
	Let $p$ and $q$ be distinct primes and $\alpha:\Z_q\times \Z_p\to \Z_p$ be a continuous action by automorphisms. Then  $\Z_p\rtimes_{\alpha} \Z_q$ is not \HLM.
\end{lemma}
\begin{proof}
	In case  $K=\ker\alpha$ is trivial, then $\Z_q$ is  algebraically isomorphic to a subgroup of $\Aut(\Z_p)\cong \Z_p^* $.
	This is impossible since $\Z_q$ is $p$-divisible,  while $\Z_p^*= C_p F_p$ contains no infinite $p$-divisible subgroups.
	Hence, $K$ is a non-trivial closed subgroup of $\Z_q$, so   isomorphic to $\Z_q$ itself.
	By Fact \ref{fac:ext}, the group $\Z_p\rtimes_{\alpha} K\cong \Z_p\times \Z_q$ is not \HLM.
\end{proof}

The following theorem provides a criterion for hereditary minimality in terms of properties of the closed subgroups of a  compact solvable group. We are going to use it in Example \ref{ex:padic:rtimes:F} to check that the groups $K_{p,F}$ are \HM.
\begin{thm}\label{thm:charmeta}
	Let $G$ be a  compact solvable group. The following conditions are equivalent:
	\ben
	\item $G$ is \HM;
	\item there exists a prime $p$, such that for every infinite closed subgroup $H$ of $G$ either $Z(H) = \{e\}$  or $H \cong \Z_p$;
	\item there exists a prime $p$, such that for every infinite closed subgroup $H$ of $G$ either $Z(H) = \{e\}$ or $Z(H) \cong \Z_p$.
	\een
\end{thm}
\begin{proof} Since the assertion of the theorem is trivially true for finite or abelian groups, we can  assume that $G$ is infinite and solvable of class $n>1$.
	
$(1) \Rightarrow (2)$: By  Lemma \ref{lem:meta}, $B=\overline{G^{(n-1)}}\cong\Z_p$ for some prime $p$.
Let $H$ be an infinite closed subgroup of $G$ such that  $Z(H)$ is non-trivial. By Theorem C, $H\cong \Z_q$ for some prime $q$. If $q\neq p$, then $H\cap B$ is trivial. Since $B$ is normal subgroup of $G$ we deduce that $B\rtimes H \cong  \Z_p\rtimes \Z_q$ is \HM, contradicting Lemma \ref{lem:difp}.

$(2)\Rightarrow (3)$: Trivial.

$(3) \Rightarrow (1)$: Let $H$ be an infinite subgroup of $G$. We will prove that $H$ is minimal. If $H$ is abelian, then $Z(H)=H\leq Z(\overline {H})$ and by $(3)$ we deduce that $Z(\overline {H})$  is  isomorphic to $\Z_p$, so, in particular, its subgroup $H$ is minimal by Prodanov's theorem. Hence, we can  assume that $H$ is non-abelian, so solvable of class $m$ where $n\geq m>1$, and we have to show that it is essential in
$\overline H$.

Consider the closed non-trivial abelian subgroup $M=\overline{H^{(m-1)}}$ of   $\overline H$. We prove that $M\cong \Z_p$.  It suffices to show, by our assumption $(3)$ that $M$ is infinite.  Assuming the contrary, let $M$  be finite. By Corollary \ref{cor:lsia},
$\overline H$ contains an infinite 
closed abelian subgroup $A$. By $(3), \ A\cong \Z_p$ and thus $A\cap M$ is trivial. As $M$ is a normal subgroup of $\overline H$, the topological semidirect product $M\rtimes_{\alpha}A$ is well defined, where $\alpha$ is the natural  action by conjugations $\alpha:A\times M\to M.$ Being an infinite group that acts on a finite group, $A$ must have a non-trivial kernel $K=\ker\alpha.$ Hence, $K\cong \Z_p$ as a closed non-trivial subgroup of $A\cong \Z_p$. It follows that $\overline H$ contains a subgroup isomorphic to $M\rtimes_{\alpha}K\cong M\times K\cong M\times \Z_p$. Let $C$ be a non-trivial cyclic subgroup of $M$, then $G$ contains an infinite closed abelian subgroup $L\cong C\times \Z_p$ that is not isomorphic to $\Z_p$, contradicting  $(3)$.

Coming back to the proof of the essentiality of $H$ in $\overline{H}$, let $N$ be  a non-trivial closed normal   subgroup of $\overline H$.
If $N_1=N\cap M$ is trivial, then $NM\cong N\times M\cong N\times \Z_p$. Following the same ideas of the first part of the proof, one can find an infinite abelian subgroup of $G$ not isomorphic to $\Z_p$, contradicting $(3)$.
Therefore $N_1$ and $H_1=H\cap M$  are non-trivial subgroups of $M\cong \Z_p$, with $N_1$ closed in $M$. By Lemma \ref{lem:cyc}, $N_1\cap H_1$ is non-trivial.

Since $N\cap H\geq N_1\cap H_1$, we conclude that $N\cap H\neq \{e\}$,  proving the essentiality of $H$ in $\overline{H}$.\end{proof}

Now we give an application of Theorem \ref{thm:charmeta} that we use in Example \ref{ex:Tn}.

\begin{lemma}\label{laaaasts:lemma}
Let $G$ be a compact solvable torsionfree group containing a closed \HM\ solvable subgroup $G_1$ of class $k$.
If $\overline{G_1^{(k-1)}}\cong\Z_p$, and $[G:G_1]=p^n$, for some $n\in \N$ and a prime $p$, then $G$ is \HM.
\end{lemma}
\begin{proof} According to Theorem \ref{thm:charmeta}, it suffices to check that for every infinite closed subgroup $H$ of $G$, either $Z(H) = \{e\}$  or $H\cong \Z_p$. Assume that $H$ is an infinite closed subgroup of $G$, and let $H_1 = G_1 \cap H$. Then there exists $r\leq n$ such that
\begin{equation}\label{dead:equation}
[H:H_1]=p^r.
\end{equation}
So, $H_1$ is an infinite closed subgroup of $G_1$.
	
	Now assume that $Z(H)$ is non-trivial and pick any $e \ne z\in Z(H)$. Then $e\ne z^{p^m} \in H_1$ for some integer  $0 < m \leq r$, by (\ref{dead:equation}). As $z^{p^m} \in Z(H_1)$, this proves that $Z(H_1) \ne \{e\}$, so $H_1\cong \Z_p$ by Theorem \ref{thm:charmeta}.
	
	Let $A=Z(H)\cap H_1$. As $z^{p^m}\in A$, it is a non-trivial closed subgroup of $H_1\cong \Z_p$, so $A$ is isomorphic to $\Z_p$ and $[H_1:A]=p^s$ for some  $s\in \N$. Then $p^{r+s} = [H:H_1] [H_1:A] = [H:A] = [H:Z(H)] [Z(H):A]$, so $H/Z(H)$ is a finite $p$-group. Thus $H$ is abelian by Lemma \ref{lem:Z1Z2}.
	
	Since the compact abelian torsionfree  group $H$ contains a subgroup isomorphic to $\Z_p$ of finite index, we deduce that $H$ itself is isomorphic to $\Z_p$.
\end{proof}

\begin{example}\label{ex:padic:rtimes:F}\label{ex:Tn}
Now we use Theorem \ref{thm:charmeta} and Lemma \ref{laaaasts:lemma} to see that the  the infinite compact metabelian groups $K_{p,F}$ and $T_n=(\Z_2,+)\rtimes_{\beta} C_2^{2^n}$ are \HM.
\begin{itemize}
\item[(a)] To  show that $K_{p,F}=\Z_p\rtimes F$ is \HM, pick an infinite closed subgroup $H$ of $K_{p,F}$. If $H$ is non-abelian, then $Z(H)$ is trivial by Lemma \ref{lem:trabel}. If $H$ is abelian, then there exists $e\ne x\in H \cap ( \Z_p\rtimes \{1\} )$ since $F$ is finite. As $H\leq C_{ K_{p,F} }(x)\leq \Z_p\rtimes \{1\}$, we obtain that $H\cong \Z_p$. By Theorem \ref{thm:charmeta}, we conclude that $K_{p,F}$ is \HM.

\item[(b)] The group $T_n$ is torsionfree, and its compact subgroup $(\Z_2,+)\rtimes_{\beta}C_2^{2^{n+1}}\cong M_{2,n+1}$ has  index $2$ in $T_n$. Moreover, $M_{2,n+1}$ is \HM\ by Example \ref{padics:rtimes:padics}, while  $M_{2,n+1}' \cong \Z_2$ by Lemma \ref{lemma:speder}. So we deduce by Lemma \ref{laaaasts:lemma} that $T_n$ is \HM.

Alternatively, since its  open subgroup $M_{2,n+1}$ is \HLM, the group $T_n$ is also \HLM\ by  Lemma \ref{lem:HLM}(2). As $T_n$ is torsionfree, Theorem B implies that $T_n$ is \HM.
\end{itemize}
\end{example}

The next result, in which we study when some semidirect products are \HM, should be compared with Proposition \ref{prop:alpha} and Lemma \ref{lem:uonique}.
 \begin{lemma}\label{lemma:kerhm}
	Let $G=(\Z_p,+)\rtimes_{\alpha} T$, where $T$ is either finite or (topologically) isomorphic to $(\Z_p,+)$. Then $G$ is \HM\ if and only if  $\alpha$ is faithful.
\end{lemma}
\begin{proof}
	If $K =\ker\alpha$ is not trivial, then  $(\Z_p,+)\times K\leq G$ is  not \HM\ by Prodanov's theorem,
	so  $G$ is also not \HM.
	
	Now assume that  $\alpha$ is faithful. If $T$ is finite, then $T$ is isomorphic to a subgroup $F$ of $F_p$, and $G\cong K_{p,F}$  by Lemma \ref{lem:uonique}. By Example \ref{ex:padic:rtimes:F}, $G$ is \HM. In case $T$ is isomorphic to $\Z_p$, then by Proposition \ref{prop:alpha} either $G\cong M_{p,n}$ or $G\cong T_n$ for some $n\in \N.$ We use Example \ref{padics:rtimes:padics} and Example \ref{ex:Tn}, respectively, to conclude that $G$ is \HM.
\end{proof}

\subsection{The general case}\label{The general case}
The next proposition offers a reduction from the general case to a specific situation that will be repeatedly
used in the sequel, very often without explicitly giving/recalling all details.

\begin{prop}\label{prop:ms}
Let $G$ be an infinite \HM\ locally compact group, which is either compact or locally solvable.

If $G$ has a non-trivial normal solvable subgroup, then $G$ is metabelian. In particular, it has a normal subgroup $N\cong \Z_p$, such that $N = C_G(N)$, and there exists a monomorphism
\[
j: G/N \hookrightarrow \Aut(N)\cong \Aut(\Z_p) \cong \Z_p \times F_p.\eqno{(\dag)}
\]
\end{prop}
\begin{proof}
Let $A$ be a non-trivial normal solvable subgroup of $G$, and assume that $n > 0$ is the solvability class of $A$. Then $A^{(n-1)}$ is non-trivial and abelian, it is characteristic in $A$, so normal in $G$. Hence, we can assume $A$ to be abelian.

Now note that $G$ is non-discrete by Lemma \ref{lem:tnlf}, so $A$ is infinite by Proposition \ref{prop:ndhmlc}.

Let $N_1 = C_G(A)$, and note that also $N_1$ is normal in $G$. As $A$ is abelian, we have $A \leq N_1$, and indeed $A \leq Z(N_1)$, so  $Z(N_1)$ is infinite. Moreover, $N_1$ is closed in $G$, so it is hereditarily minimal, and  locally compact.
Since $N_1$ is also either compact or locally solvable, we have $N_1 \cong \Z_p$ for some prime $p$ according to Theorem C.

By induction, define $N_{n+1} = C_G(N_n)$ for $n \geq 1$, and similarly prove that $ N_n \cong\Z_p $ is normal in $G$ for every $n$. Then $H = \displaystyle{ \overline{ \bigcup_{n} N_n } }$ is abelian, locally compact, hereditarily minimal, so also $H \cong \Z_p$ by Corollary \ref{cor:spr}. Then $[H: N_1]$ is finite, so the ascending chain of subgroups $\{N_n\}_n$ stabilizes and there exists  $n_0\in \N_+$ such that $H = N_{n_0} = C_G (N_{n_0}) \cong \Z_p$ is normal in $G$.

Let $N = N_{n_0}$. Then $G$ acts on $N$ by conjugation, and the kernel of this action is $N=C_G(N)$. Hence $G/C_G(N)=G/N$  is isomorphic to a subgroup of $\Aut(N)$ via a monomorphism $j$ as in $(\dag)$. Since $N\cong \Z_p$, we have $\Aut(N)\cong \Aut(\Z_p) \cong \Z_p \times F_p$. This implies that $G/N$ is abelian, so $G$ is metabelian.
\end{proof}

Let us see that the assumption `compact or locally solvable' cannot be removed in Proposition \ref{prop:ms}.
Consider the countable discrete group $G = \Z(2) \times T$ from Example \ref{ex:tarski}. The hereditarily minimal group  $G$ has a normal abelian subgroup, but it is not compact, nor even locally solvable.

In the rest of this section $G$ is a \HM \ locally compact, solvable group, and $N$ and $j$ are as in Proposition \ref{prop:ms}.
\begin{prop}\label{prop:tri}
	
If $x\in G\setminus N$, then $H_x=\overline{\langle x\rangle}$ trivially meets $N$. Moreover, if $x$ is non-torsion, then $G_1=N \cdot H_x \cong N\rtimes H_x$ is isomorphic to  either $M_{p,n}$ or $T_n$, for some prime $p$ and $n\in \N$.
\end{prop}
\begin{proof} Clearly, we can consider only the case when $x$ is non-torsion, as $N$ is torsionfree. Hence, $H_x\cong \Z_q$ for some prime $q$ by  Lemma \ref{Prod+Steph:lemma}.
If $q \neq p$, then $H_x\cap N= \{e\}$  and  $G\geq N\cdot H_x\cong  N\rtimes H_x$, where the action is the conjugation in $G$. Moreover, this semidirect product is also isomorphic to $\Z_p\rtimes_{\alpha} \Z_q$ for some action $\alpha$, contradicting   Lemma \ref{lem:difp}. So we deduce that $H_x\cong \Z_p$.
	
As $G$ is \HM\ so is its subgroup $C$. In particular, $C$ is essential in its closure $H_x$. Since $N$ is normal in $G$, it follows that $H_x \cap N$ is trivial if and only if $C\cap N$ is trivial.

	Assume by contradiction that $C\cap N$ is non-trivial.
		Let $\varphi_x: G\to G$ be the conjugation by $x$,
	and note that $\varphi _x\upharpoonright_N\neq Id_N$, as $x\notin N=C_G(N)$. Obviously, $\varphi_x\upharpoonright _C = Id_C$, so $\varphi_x\upharpoonright _{H_x} = Id_{H_x}$, and in particular $\varphi_x\upharpoonright _{N\cap H_x} = Id_{N\cap H_x}$.
	 As $N\cap H_x$ is a non-trivial closed subgroup of $N\cong \Z_p$, we have  $N\cap H_x= p^k N$ for an integer $k$. Take an arbitrary element $t\in N$, so $p^kt\in N\cap H_x$ and $\varphi_x(p^kt) = p^kt$.
	On the other hand, $\varphi_x(p^kt) = p^k\varphi_x(t)$,
	which means $\varphi_x(t)=t$, as $N$ is torsionfree. So $\varphi_x\upharpoonright_N= Id_{N}$, a contradiction.

 Thus, $G_1=N \cdot H_x \cong N\rtimes H_x \cong \Z_p\rtimes_{\alpha} \Z_p$, for a faithful action $\alpha$ by Lemma \ref{lemma:kerhm}.  Finally, apply Proposition \ref{prop:alpha}.
\end{proof}

In the following result, we consider the canonical projection $q: G \to G/N$, and we study how the torsion elements of $G$ are related to those of $G/N$. To this end, recall that $G/N$ is (algebraically isomorphic to) a subgroup of $\Aut(N)\cong \Aut(\Z_p)\cong \Z_p \times F_p$, so  the torsion subgroup of $G/N$ is isomorphic to a subgroup of $F_p$.
\begin{prop}\label{prop:abc}

	\ben[(a)]
	\item Let $x \in G \setminus N$. Then $C = \langle x\rangle \cong \langle q(x)\rangle$, so $x$ is torsion if and only if $q(x)$ is torsion. In the latter case, $C$ is isomorphic to a subgroup of $F_p$.

	\item
	$G/N$ is torsionfree if and only if $G$ is torsionfree.
	
	\item Assume $G$ has torsion, and fix a torsion element $x_0$ of $G$ of maximum order $m$. Then the subgroup $S_p= N\cdot\langle x_0\rangle$ contains all torsion elements of $G$, and indeed
	$t (G) = t(S_p) = S_p \setminus N$. Furthermore, $S_p\cong K_{p,F}$, where $F$ is the subgroup of $F_p$ isomorphic to $ \langle x_0\rangle.$ \een
\end{prop}
\begin{proof}
	(a) By Proposition \ref{prop:tri}, the map $q$ restricted to $C$ induces an isomorphism $C \cong \langle q(x)\rangle$. For the last part, use the fact that the torsion subgroup of $G/N$ is isomorphic to a subgroup of $F_p$.

	(b) If $G$ has a non-trivial torsion element $g$, then $g \notin N\cong \Z_p$. So, $q(g)$ is a non-trivial torsion element of $G/N$.
	  By item (a), if $q(g)$ is a non-trivial torsion element of $G/N$, then
	  $g$ is a  non-trivial torsion element of $G$.
	
	(c) Using item (a) and the fact that $t(G/N)$ is cyclic  we deduce
	that $t(G/N)=t(S_p/N)=\langle q(x_0) \rangle$. Moreover,  we have
	$$t(G)=q^{-1}(t(G/N))\setminus N= q^{-1}(t(S_p/N))\setminus N=S_p\setminus N.$$
As $S_p$ is \HM\ (being a subgroup of $G$), Lemma \ref{lemma:kerhm} implies  that  $S_p\cong N\rtimes_{\alpha} \langle x_0 \rangle $, for some faithful action $\alpha$. The isomorphism $S_p\cong K_{p,F}$ now follows From Lemma \ref{lem:uonique}.
\end{proof}

By Proposition \ref {prop:ms} we may identify $G/N$  with a subgroup of $\Z_p\times F_p$ via the monomorphism $j$. For $p>2$  the following dichotomy holds.
\begin{prop}\label{prop:dich}
	If $p>2$, then  either $j(G/N)\leq \Z_p$ or $j(G/N)\leq F_p$.
\end{prop}
\begin{proof} By contradiction, let $(x,t)\in j(G/N)$ be such that $x\neq 0_{\Z_p}$ and $t\neq e_{F_p}$. Let $q:G\to G/N$ be the canonical map and pick $z\in G$ such that $q(z)=(x,t).$  Then  $C =  \langle z\rangle$  misses  $N$, as $(x,t)$ is non-torsion. Furthermore, $C_1=\overline{C}\cong \Z_q$ for some prime $q$. As  $N_1 = N \cap C_1$ misses $ C$ and as  $C$  is minimal
	(since $G$ is HM), we deduce that the closed subgroup  $N_1$ of $ C_1$ must be trivial by the Minimality Criterion.
	Hence, the subgroup  $N\rtimes C_1\cong \Z_p\rtimes \Z_q$ is \HM. This is possible only if $q = p$ by Lemma \ref{lem:difp}. Thus, $C_1\cong \Z_p$, so  also $H=q(C_1)$ is
	(even topologically) isomorphic to $\Z_p$. Since $\Z_p$ is $(p-1)$-divisible we have $H=(p-1)H$. Now let $\pi_2:\Z_p\times F_p \to F_p$ be the projection on the second coordinate. We have $\pi_2(H)=\pi_2( (p-1)H)=\{e_{F_p}\}.$ On the other hand, $e_{F_p}\neq t\in \pi_2(H)$, a contradiction.
\end{proof}

\subsection{Torsionfree case}\label{Torsionfree case}

The next theorem classifies the locally compact solvable HM groups, that are also torsionfree. The non-torsionfree groups are considered in Theorem \ref{prop:nonfree}.
\begin{thm} \label{thm:free}	Let $G$ be an infinite locally compact solvable torsionfree  group, then the following conditions are equivalent:
	\ben
	\item  $G$ is \HM;
	\item $G$ is topologically isomorphic to one of the following groups:
	\ben [(a)]
	\item  $\Z_p$ for some prime $p$;
	\item $M_{p,n}=\Z_p \rtimes C_p^{p^n}$, for some prime $p$ and  \ $n\in \N$;
	\item $T_n=(\Z_2,+)\rtimes_{\beta}C_2^{2^n} $, for some $n\in \N$.
	\een \een
\end{thm}

 \begin{proof} $(2)\Rightarrow (1):$ If $G\cong \Z_p$, then it is \HM\ by Prodanov's theorem. If $G\cong M_{p,n}$, then it is \HM\ by Example \ref{padics:rtimes:padics}(a).
If $G\cong T_n$ for some $n$, then $G$ is \HM\ by Example \ref{padics:rtimes:padics}(b).

$(1)\Rightarrow (2)$:  If $G$ is abelian, then $G\cong \Z_p$ for some prime $p$, by Prodanov's theorem.
	
Assume from now on that $G$ is non-abelian, so $G \ne N$ and let $x_1\in G \setminus N$. By  Proposition \ref{prop:tri},
$N$ trivially meets $\overline{\langle x_1\rangle}$, and the subgroup $G_1 = N\cdot \overline{\langle x_1\rangle}$ is isomorphic to $\Z_p \rtimes_\alpha \Z_p$,
for some faithful action $\alpha.$ Clearly, $G' \leq N\leq G_1$, so $G_1$ is normal.
Moreover, $G_1/N$ is topologically isomorphic to $\overline{\langle x_1\rangle}$, hence to $\Z_p$ by Lemma \ref{Prod+Steph:lemma}.
	
Now we claim that there is a monomorphism $f:G/N \to \Z_p$. Since $G$ is torsionfree, $G/N$ is also torsionfree by Proposition \ref{prop:abc}. Moreover, according to Proposition \ref{prop:ms}, there exists a group monomorphism $j:G/N \to \Z_p\times F_p$, so also  $W=j(G/N)$ is torsionfree. If $p> 2$, then $W\leq\Z_p$ by Proposition \ref{prop:dich}, so we simply take $f=j$. If $p=2$, the subgroup $2W$ of $W$ is isomorphic to a subgroup of $\Z_2\times \{e_{F_2}\}$. We define $\phi :\Z_2\times F_2\to \Z_2$ by $\phi(x)= 2x$. Then $\phi\upharpoonright_W$ is an isomorphism between $W$ and $2W$ since $W$ is torsionfree. Hence, $f=\phi\circ j$ is a monomorphism $G/N \to \Z_2$.

Equipping the codomain of $f$ with the $p$-adic topology, we now prove that $f$ is a topological isomorphism onto its range.
As $G_1/N$ is topologically isomorphic to $\Z_p$, the restriction $f\upharpoonright_{G_1/N }:G_1/N \to \Z_p$ is continuous.
In particular, $f(G_1/N)\cong G_1/N$ is a non-trivial compact subgroup of $\Z_p$. Hence, $f(G_1/N)$ is open and $G_1/N \cong f(G_1/N) = p^k\Z_p$ for some $k\in \N$. We deduce that $f(G/N)$ is also open, so $f(G/N) = p^t\Z_p$ for some $t\in \N$. Since $N$ is a common closed normal subgroup of both $G$ and $G_1$, we get
\[
G/N \Big/ G_1/N \cong \ f(G/N)\Big/f(G_1/N)\cong p^t\Z_p\big/p^k\Z_p\cong \Z(p^{t-k}).
\]
 As $G_1/N$ is open in $G/N$ which is compact,  we obtain that $f:G/N\to f(G/N)$  is a topological isomorphism.  In particular,  $G/N$ is  topologically isomorphic to $\Z_p$.
	
	For every $x \in G\setminus N$, let
	$G_x$ be the closed subgroup of $G$    generated by $N$ and $x$.	 Consider the families $\mathcal F_1=\{G_x\}_{x\notin N}$  and $\mathcal F_2=\{q(G_x)\}_{x\notin N}$,
	where $q: G \to G/N$ is the canonical map.
	As $G/N\cong \Z_p$, the family $\mathcal F_2$ is  totally ordered by inclusion and every member of $\mathcal F_2$ has finitely many successors in $\mathcal F_2$.
	Since all the subgroups $G_x$ contain $N$, it follows that $\mathcal F_1$  is also  totally ordered and has the same property.
	In particular, the members in  $\mathcal F_1$ that contain $G_1$  are finitely many, so  there is a maximal member among them
and it coincides with $G$ as $\displaystyle \bigcup_{x\notin N} G_x=G$.  Finally, apply Proposition \ref{prop:tri}.
\end{proof}

\subsection{Non-torsionfree case}\label{Non-torsionfree case}

 Given an action $\alpha:(\Z_2,+)\times (\Z_2,+)\to (\Z_2,+)$ and $(b,c)\in \Z_2\times \Z_2$  we write  shortly $b(c)$ in place of $\alpha(b,c)$.
A map $\eps:(\Z_2,+)\to (\Z_2,+)$ is called a {\em crossed homomorphism}, if
	\begin{equation*}
	\eps(b_1+b_2)=\eps(b_1)+b_1(\eps(b_2)), \forall b_1, b_2\in \Z_2.
	\end{equation*}

The following technical result is used in the proof of Theorem \ref{prop:nonfree}, in the special case of $p=2$.

\begin{prop}\label{prop:sknmin} Let $\alpha:(\Z_2,+)\times (\Z_2,+)\to (\Z_2,+)$ be a continuous faithful action by automorphisms and let $f$ be
an automorphism of $M_{2,\alpha}= (\Z_2,+)\rtimes_{\alpha} (\Z_2, +)$ of order $2$. If the automorphism $\bar f: M_{2,\alpha}/M_{2,\alpha}' \to M_{2,\alpha}/M_{2,\alpha}'$ induced by  $f$ is identical,  then the group $G=((\Z_2,+)\rtimes_{\alpha} (\Z_2, +))\rtimes F$, where $F=\{f,Id\}$, is not \HM.

\end{prop}
\begin{proof}

	Since the group $F$ is finite we deduce that  the topological semidirect product $G=M_{2,\alpha}\rtimes F$ is well defined.  According to Theorem B, it suffices to show that $G$ does not satisfy ($\mathcal{C}_{fn}$). Consider the element $g_0:= (e,f)$ of $G$, note that it has order $2$ in $G$, and let $P=\langle g_0\rangle$. It is sufficient to find an infinite compact abelian subgroup $H$ of $G$ containing $P$. To this end, we shall provide a non-torsion element $g \in G$, commuting with $g_0$, and let $H=\overline{\langle g, g_0\rangle}$.
	In view of the semidirect product structure of $G$ and since $M_{2,\alpha}$ is  torsion free, this means to find a fixed point of $f$, i.e., an element $e\ne x\in M_{2,\alpha}$ with $f(x)  = x$,
	 since in this case $g= (x,f) \in G $ works, as $g$ commutes with $g_0$.

First of all, we need to find an explicit form of $f$. For the sake of brevity put $N = \Z_2 \times \{0\}$ and $T = \{0\}\times \Z_2$. Our hypothesis that $\bar f$ is identical allows us to claim that also the automorphism $T\to T$,	induced by $\bar f$, of the quotient $T\cong M_{2,\alpha}/N$ is identical, so
we can	write
$f(0,b)=(\eps(b),b)$ for every $b\in \Z_2$, where  $\eps : \Z_2\to \Z_2$ is a continuous crossed homomorphism, as $f$ is a continuous homomorphism.
Moreover, since $f$ is an automorphism of $M_{2,\alpha}$ and  $N\geq M_{2,\alpha}'$ it follows that  $f\restriction_{N}$ is an automorphism of $N$.
In fact, as $N\cong \Z_2$, $f\restriction_{N}$ is the multiplication by some  $m\in \Z_2^*$, and $m\in F_2=\{\pm 1\}$ as $o(x_0)=2$.
Hence, we can write $f$ as follows:
\begin{equation}\label{888}f(a,b)=f (a,0)f(0,b)=(ma,0)(\eps(b),b)=(ma+\eps(b),b).
\end{equation}
  Assume that $m= 1$, then (\ref{888}) and $f^2 = Id$ give $2 \eps(b)= 0$, hence $\eps(b)= 0$, for every $b = 0$.
Take any $b\ne 0$ in $\Z_2$ and put $x =(0,b)$. Then obviously  $f(x)  = x$ and we are done.

In case $m=-1$. Take any $b\neq0$ in $\Z_2$, so $2b\neq 0$. Recall that $\Aut (\Z_2) \cong \Z_2^* = \Z_2 \setminus 2 \Z_2 = 1 + 2\Z_2$, so  the action $\alpha$ gives $m_b\in \Z_2^* = 1 + 2\Z_2$ such that $b(\eps(b))=m_b\eps(b)$. Hence,
	\[
	\eps(2b) = \eps(b) + m_b \eps(b) = (1+m_b)\eps(b)\in 2\Z_2,
	\]
	which means there exists $s\in \Z_2$ such that $\eps(2b)=2s$. Consider the element $x = (s,2b) \in M_{2,\alpha}$, and observe that it is not torsion as
	$2b\neq 0$. Finally, $x$ is a fixed point of $f$. Indeed, $f(x) = ( -s+\eps(2b), 2b) = (-s + 2s, 2b) = (s , 2b) = x$.
\end{proof}

Here follows the counterpart of Theorem \ref{thm:free} for non-torsionfree groups, which concludes the classification of the locally compact solvable HM groups given in Theorem D.
\begin{thm}\label{prop:nonfree}
	Let $G$ be an infinite locally compact solvable non-torsionfree group. Then the following conditions are equivalent:
	\ben [(a)]
	\item  $G$ is \HM;
	\item $G$ is topologically isomorphic to $K_{p,F}$ for some prime $p$ and a non-trivial subgroup $F$ of $F_p$.
	 \een
\end{thm}
\begin{proof}
 $(b)\Rightarrow (a)$: If $G\cong K_{p,F}$, then  $G$ is \HM\ by Example \ref{ex:padic:rtimes:F}.

 $(a)\Rightarrow (b)$: In the notation of Proposition \ref{prop:abc},	let $S_p=N\cdot \langle x_0\rangle$, where $x_0$ is a torsion element of $G$ of maximal order. Moreover,  $t(G)\subseteq S_p$ and $S_p\cong K_{p,F}$ for some non-trivial $F\leq F_p$. So, it suffices to show that $G=S_p$.

By contradiction, pick $x\in G\setminus S_p$ and observe that $x\notin N$ and $x$ is non-torsion. The subgroup  $T=\overline{\langle x \rangle}$ is isomorphic to $\Z_q$ for some prime $q$ by Prodanov's theorem. Moreover,
\begin{equation}\label{999}
T\cap N=\{e\},
\end{equation}
by Proposition  \ref{prop:tri} If $p>2$, then $G/N$ is isomorphic to a subgroup of $F_p$ by Proposition \ref{prop:dich} and Proposition \ref{prop:abc}(b). Hence $[G:N]<\infty$, and this contradicts (\ref{999}), as $T$ is infinite.

Now assume that $p=2$. In view of the normality of $N\cong \Z_p$ in $G$ and (\ref{999}), we can apply Lemma \ref{lem:difp} to deduce that $q=2$ as well. So
 $L=NT$ is isomorphic to $N\rtimes T\cong M_{2,\alpha}$, where $\alpha$ is a faithful action, by Lemma \ref{lemma:kerhm}.  In the sequel we identify $L = NT\cong N\rtimes T$ and $M_{2,\alpha}$. Since $G'\leq N \leq L$ by Proposition \ref{prop:ms},  it follows that $L$ is normal in $G$.
As $x_0$ has order two and $L$ is torsionfree, we obtain  $$G\geq L\rtimes \langle  x_0 \rangle \cong (\Z_2\rtimes_{\alpha} \Z_2)\rtimes_{\beta}F_2,$$ where $\beta$  is a faithful action. Indeed, if $\ker\beta$ is non-trivial,  then $\ker\beta=F_2$ and we deduce that
$$
(\Z_2\rtimes_{\alpha} \Z_2)\rtimes_{\beta}F_2\cong (\Z_2\rtimes_{\alpha} \Z_2)\times F_2,
$$ contradicting the fact that $G$ is \HM.
	
The action $\beta$ coincides with the action $\phi:L\to L$ by conjugation by $x_0$, and $o(\phi)=2$ as $o(x_0)=2$ and $x_0\notin C_G(N)=N\leq L$. Clearly $\phi(L')=L'$, as $L'$ is a normal subgroup of $G$,  so  $\phi$ induces an automorphism $\bar \phi: L/L' \to L/L'$. Since $\bar \phi $ is still an internal automorphism of $L/L'$ and this quotient is abelian,	 we deduce that $\bar \phi = Id_{L/L'}$. Now we can apply  Proposition \ref{prop:sknmin} to $f = \phi$ to deduce that $(\Z_2\rtimes_{\alpha} \Z_2)\rtimes F_2$ is not \HM, contradicting  the fact that $G$ is \HM.	
\end{proof}

\begin{corol}
Let $G$ be an infinite hereditarily minimal locally compact group, which is either compact or locally solvable.

Then $G$ contains a subgroup $K \cong \Z_p$ for some prime $p$, such that its normalizer $N_G(K)$ is isomorphic to one of the groups in $(a)$, $(b)$ or $(c)$, in Theorem D.
\end{corol}
\begin{proof}
As $G$ is non-discrete by Lemma \ref{lem:tnlf}, the existence of $K$ is guaranteed by Proposition \ref{prop:ndhmlc}(2). The subgroup $N = N_G(K)$ is closed in $G$ and contains $K$, so $N$ satisfies all the properties of $G$ listed in the statement. As $K$ is normal in $N$, we can apply Proposition  \ref{prop:ms} to conclude that $N$ is metabelian. Finally, Theorem D applies to $N$.
\end{proof}

\subsection{Hereditarily  totally minimal topological  groups}\label{Hereditarily  totally minimal topological  groups}

 In this subsection we classify the infinite locally compact solvable hereditarily totally minimal groups. We also provide an equivalent condition for a non-discrete locally compact  group with non-trivial center to be HTM.
\begin{defi}\label{def:htm}
Call a topological group $G$ hereditarily totally minimal, if every subgroup of $G$ is  totally  minimal.
\end{defi}

The following concept has a key role in the Total Minimality Criterion (see Fact \ref{fac:TMC}).
\begin{defi}
A subgroup $H$ of a topological group $G$ is totally dense if for every closed normal subgroup $N$ of $G$ the intersection $N\cap H$ is dense in $N$.
\end{defi}

\begin{fact} \label{fac:TMC}
  \cite[Total Minimality Criterion]{DP} Let $H$ be a dense subgroup of a topological group $G$. Then $H$ is totally minimal if and only if $G$ is totally minimal and $H$ is totally dense  in $G$.
\end{fact}

As hereditarily totally minimal groups are HM, the groups $K_{p,F}, M_{p,n}$ and $T_n$ are the only locally compact solvable groups we need to consider, according to Theorem D. In the next proposition we prove that the groups $K_{p,F}$ are hereditarily totally minimal.

\begin{prop}\label{prop:kpfhtm}
Let $p$ be a prime and $F\leq F_p$. Then $K_{p,F}$ is hereditarily totally minimal.
\end{prop}
\begin{proof}
Let $H$ be an infinite subgroup of $K_{p,F}$ and we have to prove that $H$ is totally minimal. As $K_{p,F}$ is compact,  it suffices to show that $H$ is totally dense in $\overline{H}$, by  Fact \ref{fac:TMC}. To this aim, let $N$ be a non-trivial closed normal subgroup of $\overline{H}$. As $\overline{H}$ is an infinite compact HM group, its subgroup $N$ is infinite by Proposition \ref{prop:ndhmlc}, so it intersects non-trivially the subgroup $\Z_p\rtimes \{1\}$. Moreover, there exists $n\in \N$ such that $N\geq N\cap (\Z_p\rtimes \{1\})= p^n\Z_p\rtimes \{1\}$. This implies that $N$ is open in $K_{p,F}$, so in $\overline{H}$, hence $\overline{H\cap N}=N$.
\end{proof}

The following Lemma will be used in the proof of Proposition \ref{prop:mpntn}.
\begin{lemma}\label{lemma:strong}
If $G$ is a hereditarily totally minimal group, then all quotients of  $G$ are hereditarily totally minimal.
\end{lemma}
\begin{proof}
Let $N$ be a closed normal subgroup of  $G$ and let  $q:G \to G/N$ be the quotient map. Take a subgroup $D$
of $G/N$  and let $D_1 = q^{-1}(D)$. Now we prove that $D$ is totally minimal. Consider the restriction $q': D_1 \to D$, clearly, $q'$ is a continuous surjective  homomorphism. Since $D_1$ is totally minimal by our hypothesis on $G$, we obtain that $q'$ is open. Moreover, being a quotient group of $D_1$, we deduce that $D$ itself is totally minimal.
\end{proof}

Now we show that the HM groups $M_{p,n}$ and $T_n$ are not hereditarily totally minimal.
\begin{prop}\label{prop:mpntn}
Let $p$ be a prime  and $n\in \N$, then:
\ben
\item $M_{p,n}$ is not hereditarily totally minimal;
\item $T_n$ is not hereditarily totally minimal.
\een
\end{prop}
\begin{proof}
(1): Assume for a contradiction that $M_{p,n}$ is hereditarily totally minimal for some prime $p$ and $n\in \N$. Then the quotient group $M_{p,n} / M_{p,n}'$ is hereditarily totally minimal by Lemma \ref{lemma:strong}.

If $p>2$, then $M_{p,n} / M_{p,n}'\cong \Z(p^{n+1})\times C_p^{p^n}$ by Lemma \ref{lemma:speder}, but this group is not even \HM\ by Fact \ref{TeoP}. In case $p=2$, as $M_{2,n} / M_{2,n}'\cong \Z(2^{n+2})\times C_2^{2^n}$, we get a similar contradiction.

(2): By Lemma \ref{commTn} the quotient group $T_n / T_n'$ of $T_n$ is isomorphic to $\Z(2) \times C_2^{2^n}$, which is not hereditarily totally minimal. Alternatively, one can note that $T_n$ contains $M_{2,n+1}$, which is not hereditarily totally minimal by (1).
\end{proof}

\begin{thm}\label{thm:htmkpf}
Let $G$ be an infinite locally compact solvable group, then the following conditions are equivalent:
\ben
	\item  $G$ is hereditarily totally minimal;
	\item $G$ is topologically isomorphic to $K_{p,F}  = \Z_p \rtimes F$, where $F\leq F_p$ for some prime $p$.\een
\end{thm}
\begin{proof}
$(1)\Rightarrow(2)$: Clearly, a hereditarily totally minimal group is \HM. So, by Theorem D, $G$ is topologically isomorphic to one of the three types of groups: $K_{p,F}$, $M_{p,n}$ or $T_n$. In addition, neither $M_{p,n}$ nor $T_n$ are  hereditarily totally minimal by Proposition \ref{prop:mpntn}.
Hence, $G$ is topologically isomorphic to $K_{p,F}$ for some prime $p$ and  $F\leq F_p$.\\
$(2)\Rightarrow(1)$:  Use Proposition \ref{prop:kpfhtm}.
\end{proof}

The next fact was originally proved in \cite{EDS}, and we use it in the subsequent Theorem \ref{thm:htmfi}.
\begin{fact}\cite[Theorem 7.3.1]{DPS}\label{fac:EDS}
	If a topological group $G$ contains a compact normal subgroup $N$ such that $G/N$ is (totally) minimal, then $G$ is (resp., totally) minimal.
\end{fact}

The groups $K_{p,F}$, where $F$ is a non-trivial subgroup of $F_p$, are center-free. The next theorem deals with the case of non-trivial center.

\begin{thm}\label{thm:htmfi}
Let $G$ be a locally compact non-discrete group with non-trivial center.	 The following conditions are equivalent:
	\ben  \item $G$ is hereditarily totally minimal;
	\item every non-trivial closed subgroup of $G$ is open, $Z(G)\cong \Z_p$ for some prime $p$ and $G/Z(G)$ is hereditarily totally minimal.\een
\end{thm}
\begin{proof}
By Fact \ref{DS-theorem}, the assertion holds true in the abelian case, so we may assume that $Z(G)\ne G$.
	
	$(1)\Rightarrow (2)$: If $G$ is hereditarily totally minimal, then $G/Z(G)$ is hereditarily totally minimal, by Lemma \ref{lemma:strong}.
Applying Theorem  \ref{add:prop:thm} we conclude that every closed non-trivial subgroup of $G$ is open  and $Z(G)\cong \Z_p$ for some prime $p$.

	$(2)\Rightarrow (1)$:  Note that $G/Z(G)$ is a discrete hereditarily totally minimal group.
	
	We first prove that if $H$ is a closed subgroup of $G$, then $H$ is totally minimal. As $Z(G)\cong \Z_p$, the subgroup $H\cap Z(G)$ is compact and normal in $H$. We also have $H/(H\cap Z(G)) \cong HZ(G)/Z(G)\leq G/Z(G)$, so $H/(H\cap Z(G))$ is totally minimal since $G/Z(G)$ is hereditarily totally minimal. By Fact  \ref{fac:EDS}, $H$ is totally minimal.
	
	Now let $H$ be a  subgroup of $G$.  By Fact \ref{fac:TMC}, it remains to show that $H$ is totally dense in $\overline{H}$. To this aim, take a non-trivial closed normal subgroup $N$ of $\overline{H}$. Clearly, $N$ is also closed in $G$ and by our assumption it is open in $G$. In particular, $N$ is open in $\overline{H}$ so
	$\overline{H\cap N}=N$.
	\end{proof}

Now we recall of special case of a theorem, known as {\em Countable Layer Theorem}.

\begin{thm}\label{coro1}
\cite[Theorem 9.91]{HM}
Any profinite group $G$ has a canonical countable descending sequence
\begin{equation}\label{eq:CLT}
G = \Omega_0(G)\supseteq \ldots \supseteq  \Omega_n(G) \supseteq  \ldots \ldots
\end{equation}
 of closed characteristic subgroups of G with the following two properties:
\ben
\item $\bigcap_{n=1}^{\infty} \Omega_n(G) = \{e\}$;

\item  for each $n \in \N_+$, the quotient $\Omega_{n-1}(G)/\Omega_n(G)$ is  isomorphic to a cartesian product of simple finite groups. \een
\end{thm}

\begin{thm}\label{neww}
 Every locally compact hereditarily totally minimal group is metrizable.
\end{thm}

\begin{proof}  Let $G$ be a locally compact HTM group. By Proposition \ref{prop:ndhmlc}(1), $G$ is totally disconnected. Let $K$ be a
compact open subgroup of $G$. It is enough to show that $K$ is metrizable.
According to the Countable Layer Theorem, one has a decreasing chain of closed normal subgroups (\ref{eq:CLT})
with $\bigcap_{n=1}^{\infty} \Omega_n(K) = \{e\}$.
By Lemma \ref{lemma:strong}, the quotient group $G/ \Omega_1(K)$ is HM. Since this is a direct product
of finite groups, this is possible only if $G/ \Omega_1(K)$ is finite. Similarly, each quotient group $\Omega_n(K)/ \Omega_{n+1}(K)$ is  finite.
Hence, also  each quotient group $G/ \Omega_{n}(K)$ is  finite. This implies that each subgroup $\Omega_n(K)$ is open. Now $\bigcap_{n=1}^{\infty} \Omega_n(K) = \{e\}$ and the compactness of $K$ imply that the subgroups $\{\Omega_n(K): n\in \N\}$ form a local base at $e$. Therefore, $K$ is metrizable.
\end{proof}

\section{Open questions and concluding remarks}\label{Open questions and concluding remarks}
Recall that a group  is  hereditarily non-topologizable when it is hereditarily  totally minimal in the discrete topology.
\begin{question}\label{CC:groups}
	Are discrete \HM \ groups also hereditarily non-topologizable ?
\end{question}
In case of affirmative answer, one can deduce from the results quoted in \S \ref{discreteHM} that the class of countable discrete \HM\ groups
is stable under taking finite products. On the other hand, if $G$ and $H$ are discrete \HM\ groups such that their product is not \HM, then
one of these groups is not hereditarily non-topologizable, so provides a counter-example to Question \ref{CC:groups}.

\smallskip

In view of Theorem \ref{add:prop:thm}, the following natural question arises:
\begin{quest}\label{que:exist}
Does there exist a non-discrete hereditarily minimal locally compact group $G$ with $\{e\} \neq Z(G) \neq G$?
\end{quest}

By Theorem \ref{add:prop:thm}, a positive answer to Question \ref{que:exist} provides a group $G$ such that the center is open, and has infinite index in $G$, so $G$ is not compact (this follows also from Theorem C). So by Corollary \ref{cor:spr}, Question \ref{que:exist} has this equivalent formulation: does there exist a \HM\ locally compact group $G$ which is neither discrete nor compact,  and has non-trivial center?
In other words:

\begin{question}\label{QQ:DC}
\begin{enumerate}
\item
Does there exist a \HM\ locally compact group which is neither discrete nor compact?
\item Can such a group have non-trivial center?
\end{enumerate}
\end{question}

As already noted, a positive answer to Question \ref{que:exist} gives a positive answer to both items in Question \ref{QQ:DC}. On the other hand, a group $G$ providing a positive answer to Question \ref{QQ:DC}, if not center-free, gives also a positive answer to Question \ref{que:exist}.

\medskip

One can focus his search for an answer to Questions \ref{que:exist} and \ref{QQ:DC} using Proposition  \ref{prop:cand} below.

Theorem \ref{add:prop:thm} provides some necessary conditions for a non-discrete locally compact group $G$ satisfying $\{e\} \neq Z(G) \neq G$ to be also hereditarily minimal. Among others: being torsionfree, having the center open and isomorphic to $\Z_p$, having the quotient $G/Z(G)$ a (discrete) $p$-group. This justifies the hypotheses in the following proposition, which proves a partial converse to Theorem \ref{add:prop:thm}.

\begin{prop}\label{prop:cand}
	Let $p$ be a prime number, and $G$ be a torsionfree group such that $Z(G)\cong \Z_p$ is open in $G$. If $G/Z(G)$ is \HM, then $G$ is \HM.
\end{prop}
\begin{proof}
 Note that $G/Z(G)$ is a discrete \HM\ group by our assumption.
 By Fact  \ref{fac:EDS} and using similar arguments as in the proof of Theorem \ref{thm:htmfi} (see the implication $(2)\Rightarrow (1)$), one can prove that  every closed subgroup of $G$ is minimal.

	Now let $H$ be a subgroup of $G$, and we show that $H$ is essential in $\overline{H}.$ Since $G$ is torsionfree and $G/Z(G)$ is torsion by Lemma	\ref{lem:tnlf}, every non-trivial subgroup of $G$  meets the center non-trivially.  Let  $N$ be a closed normal non-trivial subgroup of $\overline H$. So, $H\cap Z(G)$ and $N\cap Z(G)$ are  non-trivial subgroups of $Z(G)$ and $N\cap Z(G)$ is also closed. As $Z(G)\cong \Z_p$, Lemma  \ref{lem:cyc} implies that $N\cap H\cap Z(G)$ is non-trivial. In particular, $N\cap H$ is non-trivial.
\end{proof}

If $G$ is a torsionfree group such that $Z(G)\cong \Z_p$ and $G/Z(G)$ is an infinite discrete \HM\ $p$-group, then $G$ is \HM\ by Proposition \ref{prop:cand}, and clearly $G$ is neither discrete nor compact.  So such a group provides a positive answer to Question \ref{que:exist} (so in particular also to Question \ref{QQ:DC}).
 We conjecture that such a group exists on the base of a remarkable example of a locally compact group $M$ built in \cite[Theorem 4.5]{HMOWO}, having
the following properties:
\ben
\item $Z(M)$ is open and $Z(M) \cong \Z_p$;

\item $M/Z(M)$ is a Tarski monster of exponent $p$ having $p$ conjugacy classes;

 \item every element of $M$ is contained in a subgroup, which contains $Z(M)$ and is isomorphic to $\Z_p$ (in particular, $M$ is torsionfree);

\item all normal subgroups of $M$ are central;

\item for every proper closed subgroups $H$ of $M$, $H \cong \Z_p$ and either $H \leq Z(M)$ or $Z(M) \leq H$.\een

 In particular, we see that $M$ has all of the properties listed in Theorem \ref{add:prop:thm}, but it is not clear if such an $M$ must be
HM. This will be ensured by Proposition \ref{prop:cand} if one can ensure the Tarski monster $M/Z(M)$ to be HM (i.e., hereditarily
non-topologizable). Tarski monsters  $T$ with this property were built in \cite{KOO}, so it remains to check if for such a $T$ one can
build an extension $M$ as above.

\smallskip
By Proposition \ref{prop:ms}, if $G$ is an infinite locally compact HM group, which is either compact or locally solvable, with a non-trivial normal solvable subgroup, then $G$ is metabelian. So in particular $G$ is solvable, and Theorem D applies.

\begin{question}
Can Theorem D be extended to locally solvable groups?
\end{question}

Fact  \ref{DS-theorem} shows that the abelian HM groups are second countable,  while Theorem D shows that the conclusion remains true if ``abelian'' is replaced by ``solvable and locally compact''.
On the other hand, uncountable discrete HTM groups were built in \cite{KOO}, showing that a locally compact HTM non-solvable group need not be second countable even in the discrete case. Yet this leaves open the following:

\begin{question}
Are  locally compact $\hm$ groups metrizable?
\end{question}

 Theorem \ref{neww} shows that the answer is affirmative for locally compact HTM groups. On the other hand,  Fact \ref{fac:ext} suggests that the answer can be affirmative even for locally compact HLM  groups.

 Proposition \ref{prop:mpntn} provides examples of locally compact \HM \ groups with trivial center that are not HTM.
We are not aware if a locally compact \HM \ group with non-trivial center can be non-HTM.

\subsection*{Acknowledgments}
The first-named author takes this opportunity to thank Professor Dikranjan for his generous hospitality and support.
The second-named author  is partially supported by  grant PSD-2015-2017-DIMA-PRID-2017-DIKRANJAN PSD-2015-2017-DIMA - progetto PRID TokaDyMA
 of Udine University. The third and fourth-named authors are supported by Programma SIR 2014 by MIUR, project GADYGR, number RBSI14V2LI, cup G22I15000160008 and by INdAM - Istituto Nazionale di Alta Matematica.

\end{document}